\documentclass{article}
\usepackage[utf8]{inputenc}
\usepackage{amsmath,amsthm,amssymb}
\usepackage{complexity}
\usepackage{dsfont}
\usepackage[backref,colorlinks,citecolor=blue,bookmarks=true]{hyperref}
\usepackage[capitalize]{cleveref}
\usepackage[T1]{fontenc}
\usepackage{tikz}
\usetikzlibrary{calc}
\usetikzlibrary{math}

\theoremstyle{plain}
\newtheorem{theorem}{Theorem}[section]
\newtheorem{lemma}[theorem]{Lemma}
\newtheorem{claim}[theorem]{Claim}
\newtheorem{proposition}[theorem]{Proposition}

\newtheorem{definition}[theorem]{Definition}

\newcommand*{\N}{{\mathds{N}}}

\DeclareMathOperator*{\oddgirth}{odd-girth}
\newcommand{\ecycle}[1]{{\bigcup_{i\in[#1]} A_i \times A_{i+1}}}

\newcommand{\abs}[1]{\left\vert {#1} \right\vert}
\newcommand{\cupdot}{\mathbin{\mathaccent\cdot\cup}}

\newcommand*{\linearthres}{{\delta_{\text{lin-rem}}}}
\newcommand*{\polythres}{{\delta_{\text{poly-rem}}}}

\let\poly\relax
\newcommand{\poly}{\text{poly}}

\title{ %
The Minimum Degree Removal Lemma Thresholds
}

\author{
Lior Gishboliner\thanks{Department of Mathematics, ETH, Z\"urich, Switzerland. Research supported in part by SNSF grant 200021\_196965. Email: \textbf{\{lior.gishboliner, zhihan.jin, benjamin.sudakov\}@math.ethz.ch}.}
\and 
Zhihan Jin\footnotemark[1]
\and Benny Sudakov\footnotemark[1]
}

\date{}
\begin{document}

\maketitle
\begin{abstract}
The graph removal lemma is a fundamental result in extremal graph theory which says that for every fixed graph $H$ and $\varepsilon > 0$, if an $n$-vertex graph $G$ contains $\varepsilon n^2$ edge-disjoint copies of $H$ then $G$ contains $\delta n^{v(H)}$ copies of $H$ for some $\delta = \delta(\varepsilon,H) > 0$. The current proofs of the removal lemma give only very weak bounds on $\delta(\varepsilon,H)$, and it is also known 
that $\delta(\varepsilon,H)$ is not polynomial in $\varepsilon$ unless $H$ is bipartite. Recently, Fox and Wigderson initiated the study of minimum degree conditions guaranteeing that $\delta(\varepsilon,H)$ depends polynomially or linearly on $\varepsilon$. In this paper we answer several questions of Fox and Wigderson on this topic.
\end{abstract}
\section{Introduction}
The graph removal lemma, first proved by Ruzsa and Szemer\'edi 
 \cite{ruzsatriple}, is a fundamental result in extremal graph theory. It also have important applications to additive combinatorics and property testing. The lemma states that for every fixed graph $H$ and $\varepsilon > 0$, if an $n$-vertex graph $G$ contains $\varepsilon n^2$ edge-disjoint copies of $H$ then $G$ it contains $\delta n^{v(H)}$ copies of $H$, where $\delta = \delta(\varepsilon,H) > 0$. Unfortunately, the current proofs of the graph removal lemma give only very weak bounds on $\delta = \delta(\varepsilon,H)$ and it is a very important problem to understand the dependence of $\delta$ on $\varepsilon$. The best known result, due to Fox \cite{Fox_removal_lemma}, proves that $1/\delta$ is at most a tower of exponents of height logarithmic in $1/\varepsilon$.
 Ideally, one would like to have better bounds on $1/\delta$, where an optimal bound would be that $\delta$ is polynomial in $\varepsilon$. 
 However, it is known \cite{alon2002testing} that $\delta(\varepsilon,H)$ is only polynomial in $\varepsilon$ if $H$ is bipartite. 
 This situation led Fox and Wigderson \cite{fox2021minimum} to initiate the study of minimum degree conditions which guarantee that $\delta(\varepsilon,H)$ depends polynomially or linearly on $\varepsilon$.
 Formally, let $\delta(\varepsilon, H ; \gamma)$ be the maximum $\delta \in[0,1]$ such that if $G$ is an $n$-vertex graph with minimum degree at least $\gamma n$ and with $\varepsilon n^2$ edge-disjoint copies of $H$, then $G$ contains $\delta n^{v(H)}$ copies of $H$. 
\begin{definition}
Let $H$ be a graph.
\begin{enumerate}
    \item The {\em linear removal threshold} of $H$, denoted $\linearthres(H)$,
    is the infimum $\gamma$ such that $\delta(\varepsilon, H ; \gamma)$ depends linearly on $\varepsilon$, i.e. $\delta(\varepsilon, H ; \gamma) \geq \mu \varepsilon$ for some $\mu = \mu(\gamma) > 0$ and all $\varepsilon > 0$.
    \item The {\em polynomial removal threshold} of $H$, denoted $\polythres(H)$, is the infimum $\gamma$ such that $\delta(\varepsilon, H ; \gamma)$ depends polynomially on $\varepsilon$, i.e. $\delta(\varepsilon, H ; \gamma) \geq \mu \varepsilon^{1/\mu}$ for some $\mu = \mu(\gamma) > 0$ and all $\varepsilon > 0$.
\end{enumerate}
\end{definition}

Trivially, $\linearthres(H) \geq \polythres(H)$.
Fox and Wigderson \cite{fox2021minimum} initiated the study of $\linearthres(H)$ and $\polythres(H)$, and proved that $\linearthres(K_r) = \polythres(K_r) = \frac{2r-5}{2r-3}$ for every $r \geq 3$, where $K_r$ is the clique on $r$ vertices. They further asked to determine the removal lemma thresholds of odd cycles. Here we completely resolve this question. The following theorem handles the polynomial removal threshold. 
\begin{theorem} \label{thm:poly threshold of odd cycle}
  $\polythres(C_{2k+1}) = \frac{1}{2k+1}$.
\end{theorem}

\Cref{thm:poly threshold of odd cycle} also answers another question of Fox and Wigderson \cite{fox2021minimum}, of whether $\linearthres(H)$ and $\polythres(H)$ can only obtain finitely many values on $r$-chromatic graphs $H$ for a given $r \geq 3$. \Cref{thm:poly threshold of odd cycle} shows that $\polythres(H)$ obtains infinitely many values for $3$-chromatic graphs. In contrast, $\linearthres(H)$ obtains only three possible values for $3$-chromatic graphs. 
Indeed, the following theorem determines $\linearthres(H)$ for every $3$-chromatic $H$. 
An edge $xy$ of $H$ is called {\em critical} if $\chi(H - xy) < \chi(H)$.

\begin{theorem}\label{thm:linear threshold}
For a graph $H$ with $\chi(H) = 3$, it holds that
$$
\linearthres(H) = 
\begin{cases}
    \frac{1}{2} & \text{ $H$ has no critical edge}, \\
    \frac{1}{3} & \text{ $H$ has a critical edge and contains a triangle}, \\
    \frac{1}{4} & \text{ $H$ has a critical edge and } \oddgirth(H) \geq 5.
\end{cases}
$$
\end{theorem}

Theorems \ref{thm:poly threshold of odd cycle} and \ref{thm:linear threshold} show a separation between the polynomial and linear removal thresholds, giving a sequence of graphs (i.e. $C_5,C_7,\dots$) where the polynomial threshold tends to $0$ while the linear threshold is constant $\frac{1}{4}$.

The parameters $\polythres$ and $\linearthres$ are related to two other well-studied minimum degree thresholds: the chromatic threshold and the homomorphism threshold. The chromatic threshold of a graph $H$ is the infimum $\gamma$ such that every $n$-vertex $H$-free graph $G$ with $\delta(G) \geq \gamma n$ has bounded cromatic number, i.e., 
there exists $C = C(\gamma)$ such that $\chi(G) \leq C$. The study of the chromatic threshold originates in the work of Erd\H{o}s and Simonovits \cite{ES73} from the '70s. Following multiple works \cite{AES74,Haggkvist82,Jin95,CJK97,Brandt99,Thomassen02,Thomassen07,BT10,BT11,GL11,Lyle11}, the chromatic threshold of every graph was determined by Allen et al. \cite{ABGKM13}. 

Moving on to the homomorphism threshold, we define it more generally for families of graphs. The {\em homomorphism threshold} of a graph-family $\mathcal{H}$, denoted $\delta_{\text {hom}}(\mathcal{H})$, is the infimum $\gamma$ for which there exists an $\mathcal{H}$-free graph $F = F(\gamma)$ such that 
every $n$-vertex $\mathcal{H}$-free graph $G$ with $\delta(G) \geq \gamma n$ is homomorphic to $F$. When $\mathcal{H} = \{H\}$, we write $\delta_{\text {hom}}(H)$. This parameter was widely studied in recent years \cite{luczak2006structure,OS20,LS19,ES20,Sankar}. It turns out that $\delta_{\text {hom}}$ is closely related to $\polythres(H)$, as the following theorem shows. 
For a graph $H$, let $\mathcal{I}_H$ denote the set of all minimal (with respect to inclusion) graphs $H'$ such that $H$ is homomorphic to $H'$. 
\begin{theorem}\label{thm:poly-hom}
For every graph $H$, $\polythres(H) \leq \delta_{\text{hom}}(\mathcal{I}_H)$.   
\end{theorem}
Note that $\mathcal{I}_{C_{2k+1}} = \{C_3,\dots,C_{2k+1}\}$. Using this, the upper bound in \Cref{thm:poly threshold of odd cycle} follows immediately by combining \Cref{thm:poly-hom} with the result of Ebsen and Schacht \cite{ES20} that $\delta_{\text{hom}}(\{C_3,\dots,C_{2k+1}\}) = \frac{1}{2k+1}$. The lower bound in \Cref{thm:poly threshold of odd cycle} was established in \cite{fox2021minimum}; for completeness, we sketch the proof in \nolinebreak \Cref{sec:lower bounds}. 

\medskip
The rest of this short paper is organized as follows. \Cref{sec:prelim} contains some preliminary lemmas. In \Cref{sec:lower bounds} we prove the lower bounds in Theorems~\ref{thm:poly threshold of odd cycle} and~\ref{thm:linear threshold}. 
\Cref{sec:poly} gives the proof of \Cref{thm:poly-hom}, and \Cref{sec:linear} gives the proof of the upper bounds in \Cref{thm:linear threshold}. In the last section we discuss further related problems.

\section{Preliminaries}\label{sec:prelim}
Throughout this paper, we always consider {\em labeled copies of some fixed graph $H$} and write {\em copy of $H$} for simplicity.
We use $\delta(G)$ for the minimum degree of $G$, and write $H \rightarrow F$ to denote that there is a homomorphism from $H$ to $F$. 
For a graph $H$ on $[h]$ and integers $s_1, s_2, \dots, s_h > 0$, we denote by $H[s_1,\dots,s_h]$ the blow-up of $H$ where each vertex $i \in V(H)$ is replaced by a set $S_i$ of size $s_i$ (and edges are replaced with complete bipartite graphs). 
The following lemma is standard.
\begin{lemma} \label{lem:blowup-count}
  Let $H$ be a fixed graph on vertex set $[h]$ and let $s_1, s_2, \dots, s_h \in \N$. 
  There exists a  constant $c = c(H, s_1,\dots,s_h) > 0$ such that the following holds.
  Let $G$ be an $n$-vertex graph and $V_1,\dots,V_h \subseteq V(G)$.
  Suppose that $G$ contains at least $\rho n^h$ copies of $H$ mapping $i$ to $V_i$ for all $i \in [h]$. Then $G$ contains at least $c\rho^{\frac{1}{c}} \cdot n^{s_1 + \dots + s_h}$ copies of $H[s_1,\dots,s_h]$ mapping $S_i$ to $V_i$ for all $i \in [h]$.
\end{lemma}
Note that the sets $V_1,\dots,V_h$ in \Cref{lem:blowup-count} do not have to be disjoint. 
The proof of \Cref{lem:blowup-count} works by defining an auxiliary $h$-uniform hypergraph $\mathcal{G}$ whose hyperedges correspond to the copies of $H$ in which vertex $i$ is mapped to $V_i$. By assumption, $\mathcal{G}$ has at least $\rho n^h$ edges. By the hypergraph generalization of the Kov\"ari-S\'os-Tur\'an theorem, see \cite{erdos1964extremal}, $\mathcal{G}$ contains $\poly(\rho)n^{s_1 + \dots + s_h}$ copies of $K^{(h)}_{s_1,\dots,s_h}$, the complete $h$-partite hypergraph with parts of size $s_1,\dots,s_h$. Each copy of $K^{(h)}_{s_1,\dots,s_h}$ gives a copy of $H[s_1,\dots,s_h]$ mapping $S_i$ \nolinebreak to \nolinebreak $V_i$.

Fox and Wigderson \cite[Proposition 4.1]{fox2021minimum} proved the following useful fact. 
\begin{lemma}
\label{lem:core}
    If $H \rightarrow F$ and $F$ is a subgraph of $H$, then $\polythres(H) = \polythres(F)$.
\end{lemma}

The following lemma is an asymmetric removal-type statement for odd cycles, which gives polynomial bounds. 
It may be of independent interest. 
A similar result has appeared very recently in \cite{GSV}.
\begin{lemma} \label{lem:edge disjoint short cycle implies many Ck}
  For $1 \leq \ell < k$, there exists a constant $c = c(k) > 0$ such that if an $n$-vertex graph $G$ has $\varepsilon n^2$ edge-disjoint copies of $C_{2\ell+1}$, then it has at least $c\varepsilon^{1/c} n^{2k+1}$ copies of $C_{2k+1}$. 
\end{lemma}
	\begin{proof}
		Let $\mathcal{C}$ be a collection of $\varepsilon n^2$ edge-disjoint copies of $C_{2\ell+1}$ in $G$. 
		There exists a collection $\mathcal{C}' \subseteq \mathcal{C}$ such that $|\mathcal{C}'| \geq \varepsilon n^2/2$ and each vertex $v \in V(G)$ belongs to either $0$ or at least $\varepsilon n/2$ of the cycles in $\mathcal{C}'$. Indeed, to obtain $\mathcal{C}'$, we repeatedly delete from $\mathcal{C}$ all cycles containing a vertex $v$ which belongs to at least one but less than $\varepsilon n/2$ of the cycles in $\mathcal{C}$ (without changing the graph). The set of cycles left at the end is $\mathcal{C}'$. In this process, we delete at most $\varepsilon n^2/2$ cycles altogether (because the process lasts for at most $n$ steps); hence $|\mathcal{C}'| \geq \varepsilon n^2/2$. 
		Let $V$ be the set of vertices contained in at least $\varepsilon n/2$ cycles from $\mathcal{C}'$, so $|V| \geq \varepsilon n/2$. With a slight abuse of notation, we may replace $G$ with $G[V]$, $\mathcal{C}$ with $\mathcal{C}'$ and $\varepsilon/2$ with $\varepsilon$, and denote $|V|$ by $n$. Hence, from now on, we assume that each vertex $v \in V(G)$ is contained in at least $\varepsilon n$ of the cycles in $\mathcal{C}$. 
		This implies that $|N(v)| \geq 2\varepsilon n$ for every $v \in V(G)$. 
		
		Fix any $v_0 \in V(G)$ and let $\mathcal{C}(v_0)$ be the set of cycles $C \in \mathcal{C}$ such that $C \cap N(v_0) \neq \emptyset$ and $v_0 \notin \mathcal{C}$. The number of cycles $C \in \mathcal{C}$ intersecting $N(v_0)$ is at least $|N(v_0)| \cdot \varepsilon n/(2\ell+1) \geq 2\varepsilon^2 n^2/(2\ell+1)$, and the number of cycles containing $v_0$ is at most $n$. Hence, $|\mathcal{C}(v_0)| \geq 2\varepsilon^2 n^2/(2\ell+1) - n \geq \varepsilon^2 n^2/(\ell+1)$. 
		Take a random partition $V_0,V_1,\dots,V_{\ell}$ of $V(G) \setminus \{v_0\}$, where each vertex is put in one of the parts uniformly and independently. For a cycle $(x_1,\dots,x_{2\ell+1}) \in \mathcal{C}(v_0)$ with $x_{\ell+1} \in N(v_0)$, say that $(x_1,\dots,x_{2\ell+1})$ is good if $x_{\ell+1} \in V_0$ and $x_{\ell+1-i},x_{\ell+1+i} \in V_{i}$ for $1 \leq i \leq \ell$ (so in particular $x_1,x_{2\ell+1} \in V_{\ell}$). 
		The probability that $(x_1,\dots,x_{2\ell+1})$ is good is $1/(\ell+1)^{2\ell+1}$, so there is a collection of good cycles $\mathcal{C}'(v_0) \subseteq \mathcal{C}_0$ of size $|\mathcal{C}'(v_0)| \geq |\mathcal{C}(v_0)|/(\ell+1)^{2\ell+1} \geq \varepsilon^2 n^2/(\ell+1)^{2\ell+2}$. Put $\gamma := \varepsilon^2/(\ell+1)^{2\ell+2}$. 
		By the same argument as above, there is a collection $\mathcal{C}''(v_0) \subseteq \mathcal{C}'(v_0)$ with $|\mathcal{C}''(v_0)| \geq \gamma n^2/2$ such that each vertex is contained in either $0$ or at least $\gamma n/2$ cycles from $\mathcal{C}''(v_0)$. Let $W$ be the set of vertices contained in at least $\gamma n/2$ cycles from $\mathcal{C}''(v_0)$. 
  Note that $W \cap V_0 \subseteq N(v_0)$ by definition.
  Also, each vertex in $W \cap V_{\ell}$ has at least $\gamma n/2$ neighbors in $W \cap V_{\ell}$, and for each $1 \leq i \leq \ell$, each vertex in $W \cap V_i$ has at least $\gamma n/2$ neighbors in $W \cap V_{i-1}$.
    It follows that $W \cap V_{\ell}$ contains at least 
    $\frac{1}{2}|W \cap V_{\ell}| \cdot \prod_{i=0}^{2k-2\ell-2}(\gamma n/2 - i) = \poly(\gamma) n^{2k - 2\ell}$ paths of length $2k - 2\ell - 1$. 
	We now construct a collection of copies of $C_{2k+1}$ as follows. 
    Choose a path $y_{\ell+1},y_{\ell+2},\dots,y_{2k-\ell}$ of length $2k-2\ell-1$ in $W \cap V_{\ell}$. 
    For each $i = \ell,\dots,1$, take a neighbor $y_{i} \in W \cap V_{i-1}$ of $y_{i+1}$ and a neighbor $y_{2k-i+1} \in W \cap V_{i-1}$ of $y_{2k-i}$, such that the vertices $y_1,\dots,y_{2k}$ are all different. 
    Then $y_1,\dots,y_{2k}$ is a path and $y_1,y_{2k} \in W \cap V_0 \subseteq N(v_0)$, so $v_0,y_1,\dots,y_{2k}$ is a copy of $C_{2\ell+1}$. 
    The number of choices for the path $y_{\ell+1},y_{\ell+2},\dots,y_{2k-\ell}$ is $\poly(\gamma) n^{2k - 2\ell}$ and the number of choices for each vertex $y_i,y_{2k-i+1} \in V_{i-1}$ ($i = \ell,\dots,1$) is at least $\gamma n/2$. Hence, the total number of choices for $y_1,\dots,y_{2k}$ is $\poly(\gamma)n^{2k}$. As there are $n$ choices for $v_0$, we get a total of $\poly(\gamma)n^{2k+1} = \poly_k(\varepsilon)n^{2k+1}$ copies of $C_{2k+1}$, as required. 	\end{proof}

\section{Lower bounds}\label{sec:lower bounds}
Here we prove the lower bounds in Theorems~\ref{thm:poly threshold of odd cycle} and \ref{thm:linear threshold}. The lower bound in \Cref{thm:poly threshold of odd cycle} was proved in \cite[Theorem 4.3]{fox2021minimum}. For completeness, we include a sketch of the proof:
\begin{lemma}\label{lem:C_{2k+1} construction}
    $\polythres(C_{2k+1}) \geq \frac{1}{2k+1}$.
\end{lemma}
\begin{proof}
    Fix an arbitrary $\alpha > 0$. 
    In \cite{alon2002testing} it was proved that for every $\varepsilon$, there exists a $(2k+1)$-partite graph with parts $V_1,\dots,V_{2k+1}$ of size $\alpha n/(2k+1)$ each, with $\varepsilon n^2$ edge-disjoint copies of $C_{2k+1}$, but with only $\varepsilon^{\omega(1)}n^{2k+1}$ copies of $C_{2k+1}$ in total (where the $\omega(1)$ term may depend on $\alpha$). Add sets $U_1,\dots,U_{2k+1}$ of size $(1-\alpha)n/(2k+1)$ each, and add the complete bipartite graphs $(U_i,V_i)$, $1 \leq i \leq 2k+1$, and $(U_i,U_{i+1})$, $1 \leq i \leq 2k$.
    See \Cref{fig:polynomial threshold construction for C5}.
    It is easy to see that this graph has minimum degree $(1-\alpha)n/(2k+1)$, and every copy of $C_{2k+1}$ is contained in $V_1 \cup \dots \cup V_{2k+1}$. Letting $\alpha \rightarrow 0$, we get that $\polythres(C_{2k+1}) \geq \frac{1}{2k+1}$.
    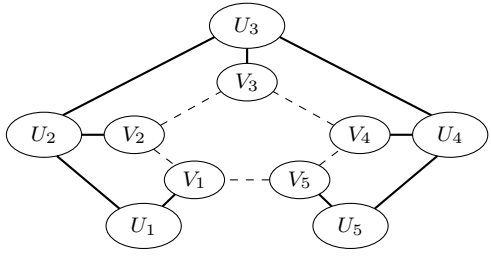
\begin{figure}
    \centering
    \begin{tikzpicture}[scale = 1]
      \coordinate (v1) at (0,0);
      \coordinate (v2) at (1.5,0.7);
      \coordinate (v3) at (3,0);
      \coordinate (v4) at (2.2,-0.6);
      \coordinate (v5) at (0.8,-0.6);
      
      \draw (v1) ellipse(0.4 and 0.25);
      \draw (v2) ellipse(0.4 and 0.25);
      \draw (v3) ellipse(0.4 and 0.25);
      \draw (v4) ellipse(0.4 and 0.25);
      \draw (v5) ellipse(0.4 and 0.25);
      
      \draw (v1) node {\footnotesize $V_2$};
      \draw (v2) node {\footnotesize $V_3$};
      \draw (v3) node {\footnotesize $V_4$};
      \draw (v4) node {\footnotesize $V_5$};
      \draw (v5) node {\footnotesize $V_1$};
      
      \coordinate (v15) at ($(v1) + ({0.4*cos(310)}, {0.25*sin(310)})$);
      \coordinate (v51) at ($(v5) + ({0.4*cos(130)}, {0.25*sin(130})$);
      \coordinate (v12) at ($(v1) + ({0.4*cos(30)}, {0.25*sin(30)})$);
      \coordinate (v21) at ($(v2) + ({0.4*cos(210)}, {0.25*sin(210)})$);
      \coordinate (v32) at ($(v3) + ({0.4*cos(150)}, {0.25*sin(150)})$);
      \coordinate (v23) at ($(v2) + ({0.4*cos(330)}, {0.25*sin(330)})$);
      \coordinate (v34) at ($(v3) + ({0.4*cos(230)}, {0.25*sin(230)})$);
      \coordinate (v43) at ($(v4) + ({0.4*cos(50)}, {0.25*sin(50})$);
      \coordinate (v54) at ($(v5) + ({0.4*cos(0)}, {0.25*sin(0)})$);
      \coordinate (v45) at ($(v4) + ({0.4*cos(180)}, {0.25*sin(180)})$);
      \draw[dashed] (v15) -- (v51);
      \draw[dashed] (v12) -- (v21);
      \draw[dashed] (v32) -- (v23);
      \draw[dashed] (v43) -- (v34);
      \draw[dashed] (v45) -- (v54);

      \coordinate (u1) at ($(v1) - ({1.2}, {0})$);
      \coordinate (u2) at ($(v2) + ({0}, {0.75})$);
      \coordinate (u3) at ($(v3) + ({1.2}, {0})$);
      \coordinate (u4) at ($(v4) + ({1.05*cos(310)}, {0.8*sin(310)})$);
      \coordinate (u5) at ($(v5) + ({1.05*cos(230)}, {0.8*sin(230)})$);
      
      \draw (u1) ellipse(0.5 and 0.3);
      \draw (u2) ellipse(0.5 and 0.3);
      \draw (u3) ellipse(0.5 and 0.3);
      \draw (u4) ellipse(0.5 and 0.3);
      \draw (u5) ellipse(0.5 and 0.3);
      
      \draw (u1) node {\footnotesize $U_2$};
      \draw (u2) node {\footnotesize $U_3$};
      \draw (u3) node {\footnotesize $U_4$};
      \draw (u4) node {\footnotesize $U_5$};
      \draw (u5) node {\footnotesize $U_1$};
      
      \coordinate (u15) at ($(u1) + ({0.5*cos(290)}, {0.3*sin(290)})$);
      \coordinate (u51) at ($(u5) + ({0.5*cos(130)}, {0.3*sin(130})$);
      \coordinate (u12) at ($(u1) + ({0.5*cos(60)}, {0.3*sin(60)})$);
      \coordinate (u21) at ($(u2) + ({0.5*cos(210)}, {0.3*sin(210)})$);
      \coordinate (u32) at ($(u3) + ({0.5*cos(120)}, {0.3*sin(120)})$);
      \coordinate (u23) at ($(u2) + ({0.5*cos(330)}, {0.3*sin(330)})$);
      \coordinate (u34) at ($(u3) + ({0.5*cos(250)}, {0.3*sin(250)})$);
      \coordinate (u43) at ($(u4) + ({0.5*cos(50)}, {0.3*sin(50})$);
      \coordinate (u54) at ($(u5) + ({0.5*cos(0)}, {0.3*sin(0)})$);
      \coordinate (u45) at ($(u4) + ({0.5*cos(180)}, {0.3*sin(180)})$);
      \draw[thick] (u15) -- (u51);
      \draw[thick] (u12) -- (u21);
      \draw[thick] (u32) -- (u23);
      \draw[thick] (u43) -- (u34);
      
      \coordinate (u11) at ($(u1) + ({0.5*cos(0)}, {0.3*sin(0)})$);
      \coordinate (v11) at ($(v1) + ({0.4*cos(180)}, {0.25*sin(180)})$);
      \coordinate (u21) at ($(u2) + ({0.5*cos(270)}, {0.3*sin(270)})$);
      \coordinate (v21) at ($(v2) + ({0.4*cos(90)}, {0.25*sin(90)})$);
      \coordinate (u31) at ($(u3) + ({0.5*cos(180)}, {0.3*sin(180)})$);
      \coordinate (v31) at ($(v3) + ({0.4*cos(0)}, {0.25*sin(0)})$);
      \coordinate (u41) at ($(u4) + ({0.5*cos(120)}, {0.3*sin(120)})$);
      \coordinate (v41) at ($(v4) + ({0.4*cos(310)}, {0.25*sin(310)})$);
      \coordinate (u51) at ($(u5) + ({0.5*cos(60)}, {0.3*sin(60)})$);
      \coordinate (v51) at ($(v5) + ({0.4*cos(230)}, {0.25*sin(230)})$);
      \draw[thick] (u11) -- (v11);
      \draw[thick] (u21) -- (v21);
      \draw[thick] (u31) -- (v31);
      \draw[thick] (u41) -- (v41);
      \draw[thick] (u51) -- (v51);
    \end{tikzpicture}
    \caption{Proof of \cref{lem:C_{2k+1} construction} for $C_5$. Heavy edges indicate complete bipartite graphs while dashed edges form the Ruzsa–Szemer\'edi construction for $C_5$ (see \cite{alon2002testing}).}
    \label{fig:polynomial threshold construction for C5}
    \end{figure}
    \end{proof}

    By combining the fact that $\polythres(C_3) = \frac{1}{3}$ with \Cref{lem:core} (with $F = C_3$), we get that $\linearthres(H) \geq \polythres(H) = \frac{1}{3}$ for every $3$-chromatic graph $H$ containing a triangle. This 
    proves the lower bound in the second case of \Cref{thm:linear threshold}. 
    Now we prove the lower bounds in the other two cases. We prove a more general statement for $r$-chromatic graphs. 
\begin{lemma}\label{thm:lower bounds}
  Let $H$ be a graph with $\chi(H) = r \ge 3$.
  Then, $\frac{3r-8}{3r-5} \le \linearthres(H) \le \frac{r-2}{r-1}$.
  Moreover, $\linearthres(H) = \frac{r-2}{r-1}$ if $H$ contains no critical edge.
\end{lemma}
\begin{proof}
  Denote $h = |V(H)|$. The bound $\linearthres(H) \leq \frac{r-2}{r-1}$ holds for every $r$-chromatic graph $H$; this follows from the Erd\H{o}s-Simonovits supersaturation theorem, see by~\cite[Section 4.1]{fox2021minimum} for the details.

  Suppose now that $H$ contains no critical edge, and let us show that $\linearthres(H) \ge \frac{r-2}{r-1}$. To this end, we construct, for every small enough $\varepsilon$ and infinitely many $n$, an $n$-vertex graph $G$ with $\delta(G) \ge \frac{r-2}{r-1} n$, such that $G$ has at most $\mathcal{O}(\varepsilon^2 n^h)$ copies of $H$, but $\Omega(\varepsilon n^2)$ edges must be deleted to turn $G$ into an $H$-free graph.
  Let $T(n,r-1)$ be the Tur\'an graph, i.e. the complete $(r-1)$-partite graph with balanced parts $V_1, \dots, V_{r-1}$.
  Add an $\varepsilon n$-regular graph inside $V_1$ and let the resulting graph be $G$. 
  We first claim that $G$ contains $\mathcal{O}(\varepsilon^2 n^h)$ copies of $H$.
  As $H$ contains no critical edge and $\chi(H) = r$, every copy of $H$ in $G$ contains two edges $e$ and $e'$ inside $V_1$.
  If $e$ and $e'$ are disjoint, then there are at most $n^2 (\varepsilon n)^2 = \varepsilon^2 n^4$ choices for $e$ and $e'$ and then at most $n^{h-4}$ choices for the other $h-4$ vertices of $H$.
  Therefore, there are at most $\varepsilon^2 n^h$ such $H$-copies.
  And if $e$ and $e'$ intersect, then there are at most $n (\varepsilon n)^2 = \varepsilon^2 n^3$ choices for $e$ and $e'$ and then at most $n^{h-3}$ choices for the remaining vertices, again giving at most $\varepsilon^2 n^h$ such $H$-copies.
  So $G$ indeed has $\mathcal{O}(\varepsilon^2 n^h)$ copies of $H$.
  
  On the other hand, we claim that one must delete $\Omega(\varepsilon n^2)$ edges to destroy all $H$-copies in $G$.
  Observe that $G$ has at least 
  $\frac{1}{2}\abs{V_1} \cdot \varepsilon n \cdot \abs{V_2} \cdot \dots \cdot \abs{V_{r-1}} = \Omega_r(\varepsilon n^r)$ copies of $K_r$, and every edge participates in at most $n^{r-2}$ of these copies. 
  Thus, deleting $c\varepsilon n^2$ edges can destroy at most $c\varepsilon n^r$ copies of $K_r$.
  If $c$ is a small enough constant (depending on $r$), then after deleting any $c\varepsilon n^2$ edges, there are still $\Omega(\varepsilon n^r)$ copies of $K_r$.
  Then, by \cref{lem:blowup-count}, the remaining graph contains $K_r[h]$, the $h$-blowup of $K_r$, and hence $H$.
  This completes the proof that $\linearthres(H) \ge \frac{r-2}{r-1}$.

  We now prove that $\linearthres(H) \geq \frac{3r-8}{3r-5}$ for every $r$-chromatic graph $H$. 
  It suffices to construct, for every small enough $\varepsilon$ and infinitely many $n$, an $n$-vertex graph $G$ with $\delta(G) \ge \frac{3r-8}{3r-5}n$, such that $G$ has at most $\mathcal{O}(\varepsilon^2 n^h)$ copies of $H$ but at least $\Omega(\varepsilon n^2)$ edges must be deleted to turn $G$ into an $H$-free graph. 
  The vertex set of $G$ consists of $r+1$ disjoint sets $V_0, V_1, V_2, \dots, V_r$, where $\abs{V_i} = \frac{n}{3r-5}$ for $i = 0, 1, 2, 3$ and $\abs{V_i} = \frac{3n}{3r-5}$ for $i = 4, 5, \dots, r$.
  Put complete bipartite graphs between $V_0$ and $V_1$, between $V_0 \cup V_1$ and $V_4 \cup \dots \cup V_r$, and between $V_i$ to $V_j$ for all $2 \le i < j \le r$.
  Put $\varepsilon n$-regular bipartite graphs between $V_1$ and $V_2$, and between $V_1$ and $V_3$. The resulting graph is $G$ (see \Cref{fig:linear threshold constructions}). 
  It is easy check that $\delta(G) \ge \frac{3r-8}{3r-5} n$.
  Indeed, let $0 \leq i \leq r$ and $v \in V_i$. If $4 \leq i \leq r$ then $v$ is connected to all vertices except for $V_i$; if $i  \in \{2,3\}$ then $v$ is connected to all vertices except $V_0 \cup V_1 \cup V_i$; and if $i \in \{0,1\}$ then $v$ is connected to all vertices except $V_2 \cup V_3 \cup V_i$. In any case, the neighborhood of $v$ misses at most $\frac{3n}{3r-5}$ vertices. 
  \begin{figure}
    \centering
    \begin{tikzpicture}[scale = 1]
      \coordinate (v1) at (0,1.3);
      \coordinate (v2) at (-1,0);
      \coordinate (v3) at (1,0);
      \coordinate (v0) at (3,0.5);
      
      \draw (v1) ellipse(0.5 and 0.3);
      \draw (v2) ellipse(0.5 and 0.3);
      \draw (v3) ellipse(0.5 and 0.3);
      \draw (v0) ellipse(0.5 and 0.3);
      
      \draw (v1) node {\footnotesize  $V_1$};
      \draw (v2) node {\footnotesize $V_2$};
      \draw (v3) node {\footnotesize $V_3$};
      \draw (v0) node {\footnotesize $V_0$};
      
      \coordinate (v11) at ($(v1) + ({0.5*cos(240)}, {0.3*sin(240)})$);
      \coordinate (v12) at ($(v1) + ({0.5*cos(300)}, {0.3*sin(300)})$);
      \coordinate (v13) at ($(v1) + ({0.5*cos(-20)},{0.3*sin(-20)})$);
      \coordinate (v21) at ($(v2) + ({0.5*cos(60)}, {0.3*sin(60)})$);
      \coordinate (v22) at ($(v2) + ({0.5}, {0})$);
      \coordinate (v31) at ($(v3) + ({-0.5*cos(60)}, {0.3*sin(60)})$);
      \coordinate (v32) at ($(v3) - ({0.5}, {0})$);
      \coordinate (v01) at ($(v0) + ({0.5*cos(160)},{0.3*sin(160)})$);
      
      \draw[dashed] (v11) -- (v21);
      \draw[dashed] (v12) -- (v31);
      \draw[thick] (v22) -- (v32);
      \draw[thick] (v13) -- (v01);
    \end{tikzpicture}
    \hspace{2cm}
    \begin{tikzpicture}[scale = 1]
      \coordinate (v1) at (0,0.8);
      \coordinate (v2) at (-1.5,0);
      \coordinate (v3) at (1.2,0);
      \coordinate (v4) at (0,-0.8);
      \coordinate (v0) at (2.8,0);
      
      \draw (v1) ellipse(0.5 and 0.3);
      \draw (v2) ellipse(0.5 and 0.3);
      \draw (v3) ellipse(0.5 and 0.3);
      \draw (v4) ellipse(0.75 and 0.45);
      \draw (v0) ellipse(0.5 and 0.3);
      
      \draw (v1) node {\footnotesize  $V_1$};
      \draw (v2) node {\footnotesize $V_2$};
      \draw (v3) node {\footnotesize $V_3$};
      \draw (v4) node {\footnotesize $V_4$};
      \draw (v0) node {\footnotesize $V_0$};
      
      \coordinate (v11) at ($(v1) + ({0.5*cos(220)}, {0.3*sin(220)})$);
      \coordinate (v12) at ($(v1) + ({0.5*cos(330)}, {0.3*sin(330)})$);
      \coordinate (v13) at ($(v1) + ({0.5*cos(-20)},{0.3*sin(-20)})$);
      \coordinate (v14) at ($(v1) + ({0}, {-0.3})$);
      \coordinate (v21) at ($(v2) + ({0.5*cos(60)}, {0.3*sin(60)})$);
      \coordinate (v22) at ($(v2) + ({0.5}, {0})$);
      \coordinate (v23) at ($(v2) + ({0.5*cos(60)}, {-0.3*sin(60)})$);
      \coordinate (v31) at ($(v3) + ({-0.5*cos(60)}, {0.3*sin(60)})$);
      \coordinate (v32) at ($(v3) - ({0.5}, {0})$);
      \coordinate (v33) at ($(v3) + ({-0.5*cos(45)}, {-0.3*sin(45)})$);
      \coordinate (v01) at ($(v0) + ({0.5*cos(140)},{0.3*sin(140)})$);
      \coordinate (v02) at ($(v0) + ({0.5*cos(220)},{0.3*sin(220)})$);
      \coordinate (v41) at ($(v4) + ({0}, {0.45})$);
      \coordinate (v42) at ($(v4) + ({0.75*cos(140)}, {0.45*sin(140)})$);
      \coordinate (v43) at ($(v4) + ({0.75*cos(50)}, {0.45*sin(50)})$);
      \coordinate (v44) at ($(v4) + ({0.75*cos(20)}, {0.45*sin(20)})$);
      
      \draw[dashed] (v11) -- (v21);
      \draw[dashed] (v12) -- (v31);
      \draw[thick] (v22) -- (v32);
      \draw[thick] (v13) -- (v01);
      \draw[thick] (v14) -- (v41);
      \draw[thick] (v23) -- (v42);
      \draw[thick] (v33) -- (v43);
      \draw[thick] (v02) -- (v44);
    \end{tikzpicture}
    
    \caption{Proof of \cref{thm:lower bounds}, $r=3$ (left) and $r=4$ (right). Heavy edges indicate complete bipartite graphs while dashed edges indicate $\varepsilon n$-regular bipartite graphs.}
    \label{fig:linear threshold constructions}
  \end{figure}
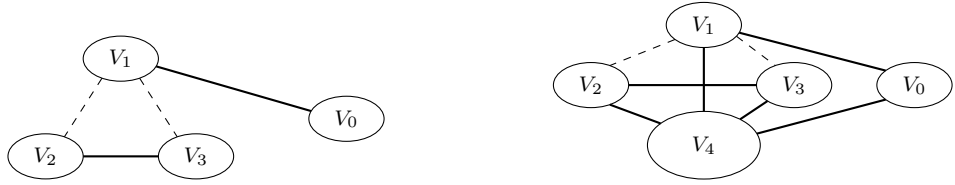
  
  We claim that $G$ has at most $\mathcal{O}(\varepsilon^2 n^h)$ copies of $H$. Indeed, observe that if we delete all edges between $V_1$ and $V_2$ then the remaining graph is $(r-1)$-colorable with coloring $V_1 \cup V_2, V_0 \cup V_3, V_4,\dots,V_r$. Hence, every copy of $H$ must contain an edge $e$ between $V_1$ and $V_2$. Similarly, every copy of $H$ must contain an edge $e'$ between $V_1$ and $V_3$. If $e,e'$ are disjoint then there are at most $n^2 (\varepsilon n)^2 = \varepsilon^2n^4$ ways to choose $e,e'$ and then at most $n^{h-4}$ ways to choose the remaining vertices of $H$. And if $e$ and $e'$ intersect then there are at most $n(\varepsilon n)^2 = \varepsilon^2 n^3$ ways to choose $e,e'$ and at most $n^{h-3}$ for the remaining $h-3$ vertices of $H$.
  In both cases, the number of $H$-copies is at most $\varepsilon^2 n^h$, as required. 
  
  Now we show that one must delete $\Omega(\varepsilon n^2)$ edges to destroy all copies of $H$ in $G$. 
  Observe that $G$ has $|V_1| \cdot (\varepsilon n)^2 \cdot  |V_4| \cdot \dots \cdot |V_r| = \Omega(\varepsilon^2 n^r)$ copies of $K_r$ between the sets $V_1,\dots,V_r$. 
  We claim that every edge $f$ participates in at most $\varepsilon n^{r-2}$ of these $r$-cliques. 
  Indeed, by the same argument as above, every copy of $K_r$ containing $f$ must contain an edge $e$ from $E(V_1,V_2)$ and an edge $e'$ from $E(V_1,V_3)$. Suppose without loss of generality that $e \neq f$ (the case $e' \neq f$ is symmetric). 
  In the case $f \cap e = \emptyset$, there are at most $n \cdot \varepsilon n = \varepsilon n^2$ choices for $e$ and at most $n^{r-4}$ choices for the remaining vertices of $K_r$, giving at most $\varepsilon n^{r-2}$ copies of $K_r$ containing $f$. 
  And if $f,e$ intersect, then there are at most $\varepsilon n$ choices for $e$ and at most $n^{r-3}$ for the remaining $r-3$ vertices, giving again $\varepsilon n^{r-2}$. 
  
  We see that deleting $c\varepsilon n^2$ edges of $G$ can destroy at most $c\varepsilon^2 n^{r}$ copies of $K_r$. Hence, if $c$ is a small enough constant, then after deleting any $c \varepsilon n^2$ edges there are still $\Omega(\varepsilon^2 n^r)$ copies of $K_r$ left. By \Cref{lem:blowup-count}, the remaining graph contains a copy of $K_r[h]$ and hence $H$. This completes the proof. 
  \end{proof}

\section{Polynomial removal thresholds: Proof of \Cref{thm:poly-hom}}
\label{sec:poly}

We say that an $n$-vertex graph $G$ is {\em $\varepsilon$-far} from a graph property $\mathcal{P}$ (e.g. being $H$-free for a given graph $H$, or being homomorphic to a given graph $F$) if one must delete at least $\varepsilon n^2$ edges to make $G$ satisfy $\mathcal{P}$. Trivially, if $G$ has $\varepsilon n^2$ edge-disjoint copies of $H$, then it is $\varepsilon$-far from being $H$-free. We need the following result from \cite{NR}. 
\begin{theorem}\label{thm:GGR}
  For every graph $F$ on $f$ vertices and for every $\varepsilon > 0$, there is $q = q_F(\varepsilon) = \poly(f/\varepsilon)$, such that the following holds. If a graph $G$ is $\varepsilon$-far from being homomorphic to $F$, then for a sample of $q$ vertices $x_1,\dots,x_q \in V(G)$, taken uniformly with repetitions, it holds that $G[\{x_1,\dots,x_q\}]$ is not homomorphic to $F$ with probability at least $\frac{2}{3}$.
\end{theorem}
  \Cref{thm:GGR} is proved in Section 2 of \cite{NR}.  
  In fact, \cite{NR} proves a more general result on property testing of the so-called 0/1-partition properties. 
  Such a property is given by an integer $f$ and a function $d : [f]^2 \rightarrow \{0,1,\perp\}$, and a graph $G$ satisfies the property if it has a partition $V(G) = V_1 \cup \dots \cup V_f$ such that for every $1 \leq i,j \leq f$ (possibly $i=j$), it holds that $(V_i,V_j)$ is complete if $d(i,j) = 1$ and $(V_i,V_j)$ is empty if $d(i,j) = 0$ (if $d(i,j) = \perp$ then there are no restrictions). 
  One can express the property of having a homomorphism into $F$ in this language, simply by setting $d(i,j) = 0$ for $i = j$ and $ij \notin E(F)$. 
  In \cite{NR}, the class of these partition properties is denoted $\mathcal{GPP}_{0,1}$, and every such property is shown to be testable by sampling $\poly(f/\varepsilon)$ vertices. This implies \Cref{thm:GGR}.
\begin{proof}[Proof of \Cref{thm:poly-hom}]
Recall that $\mathcal{I}_H$ is the set of minimal graphs $H'$ (with respect to inclusion) such that $H$ is homomorphic to $H'$. For convenience, put $\delta := \delta_{\text{hom}}(\mathcal{I}_H)$.
Our goal is to show that $\polythres(H) \leq \delta+\alpha$ for every $\alpha > 0$. 
So fix $\alpha > 0$ and let $G$ be a graph with minimum degree $\delta(G) \geq (\delta + \alpha)n$ and with $\varepsilon n^2$ edge-disjoint copies of $H$. 
By the definition of the homomorphism threshold, there is an $\mathcal{I}_H$-free graph $F$ (depending only on $\mathcal{I}_H$ and $\alpha$) such that if a graph $G_0$ is $\mathcal{I}_H$-free and has minimum degree at least $(\delta + \frac{\alpha}{2}) \cdot |V(G_0)|$, then $G_0$ is homomorphic to $F$. 
Observe that if a graph $G_0$ is homomorphic to $F$ then $G_0$ is $H$-free, because $F$ is free of any homomorphic image of $H$. It follows that $G$ is $\varepsilon$-far from being homomorphic to $F$, because $G$ is $\varepsilon$-far from being $H$-free.
Now we apply Theorem \ref{thm:GGR}. 
Let $q = q_F(\varepsilon)$ be given by Theorem \ref{thm:GGR}. 
We assume that $q \gg \frac{\log(1/\alpha)}{\alpha^2}$ and $n \gg q^2$ without loss of generality.
Sample $q$ vertices $x_1,\dots,x_q \in V(G)$ with repetition and let $X = \{x_1,\dots,x_q\}$.
By Theorem \ref{thm:GGR}, $G[X]$ is not homomorphic to $F$ with probability at least $2/3$.
As $n \gg q^2$, the vertices $x_1,\dots,x_q$ are pairwise-distinct with probability at least $0.99$. 
Also, for every $i \in [q]$,
the number of indices $j \in [q] \setminus \{i\}$ with $x_ix_j \in E(G)$ dominates a binomial distribution $\operatorname{B}(q-1, \frac{\delta(G)}{n})$. By the Chernoff bound (see e.g. \cite[Appendix A]{AlonSpencer}) and as $\delta(G) \geq (\delta + \alpha)n$, the number of such indices is at least $(\delta + \frac{\alpha}{2})q$ with probability $1 - e^{-\Omega(q \alpha^2)}$. Taking the union bound over $i \in [q]$, we get that $\delta(G[X]) \geq (\delta + \frac{\alpha}{2})|X|$ with probability at least $1 - qe^{-\Omega(q\alpha^2)} \geq 0.9$,
as $q \gg \frac{\log(1/\alpha)}{\alpha^2}$.
Hence, with probability at least $\frac{1}{2}$ it holds that $\delta(G[X]) \geq (\delta + \frac{\alpha}{2})|X|$ and $G[X]$ is not homomorphic to $F$. 
If this happens, then $G[X]$ is not $\mathcal{I}_H$-free (by the choice of $F$), hence $G[X]$ contains a copy of some $H' \in \mathcal{I}_H$. 
By averaging, there is $H' \in \mathcal{I}_H$ such that $G[X]$ contains a copy of $H'$ with probability at least $\frac{1}{2|\mathcal{I}_H|}$. 
Put $k = |V(H')|$ and let $M$ be the number of copies of $H'$ in $G$. 
The probability that $G[X]$ contains a copy of $H'$ is at most 
$M (\frac{q}{n})^k$.
Using the fact that $q = \poly_{H,\alpha}(\frac{1}{\varepsilon})$, we conclude that $M \geq \frac{1}{2|\mathcal{I}_H|}\cdot (\frac{n}{q})^k \geq \poly_{H,\alpha}(\varepsilon) n^k$. 
As $H \rightarrow H'$, there exists $H''$, a blow-up of $H'$, such that $H''$ have the same number of vertices as $H$, and that $H \subset H''$.
By \cref{lem:blowup-count} for $H'$ with $V_i = V(G)$ for all $i$, there exist $\poly_{H,\alpha}(\varepsilon) n^{v(H'')}$ copies of $H''$ in $G$, and thus $\poly_{H,\alpha}(\varepsilon) n^{v(H)}$ copies of $H$.
This completes the proof. \end{proof}

\section{Linear removal thresholds: Proof of \Cref{thm:linear threshold}}
\label{sec:linear}
Here we prove the upper bounds in \Cref{thm:linear threshold}; the lower bounds were proved in \Cref{sec:lower bounds}. 
The first case of \Cref{thm:linear threshold} follows from \Cref{thm:lower bounds}, so it remains to prove the other two cases.
We begin with some preparation. 
For disjoint sets $A_1,\dots,A_m$, we write $\ecycle{m}$ to denote all pairs of vertices which have one endpoint in $A_i$ and one in $A_{i+1}$ for some $1 \leq i \leq m$, with subscripts always taken modulo $m$. So a graph $G$ has a homomorphism to the cycle $C_m$ if and only if there is a partition $V(G) = A_1 \cup \dots \cup A_m$ with $E(G) \subseteq \ecycle{m}$.
\begin{lemma} \label{lem:chi3-decomposition}
  Suppose $H$ is a graph such that $\chi(H) = 3$, $H$ contains a critical edge $xy$, and $\oddgirth(H) \ge 2k+1$.
  Then, 
  \begin{itemize}
    \item There is a partition $V(H) = A_1 \cupdot A_2 \cupdot A_3 \cupdot B$ such that $A_1=\{x\}, A_2=\{y\}$ and $E(H) \subseteq  (A_3 \times B) \cup (\ecycle{3})$;
    \item if $k \ge 2$, there is a partition $V(H) = A_1 \cupdot A_2 \cupdot \dots \cupdot A_{2k+1}$ such that $A_1=\{x\}, A_2=\{y\}$ and $E(H) \subseteq \ecycle{2k+1}$.
    In particular, $H$ is homomorphic to $C_{2k+1}$.
  \end{itemize}
\end{lemma}
\begin{proof}
  Write $H' = H  - xy$, so $H'$ is bipartite. Let $V(H) = V(H') = L \cupdot R$ be a bipartition of $H'$. 
  As $\chi(H) = 3$, $x$ and $y$ must both lie in the same side of the bipartition.
  Without loss of generality, assume that $x, y \in L$.
  For the first item, set $A_1 = \{x\}$, $A_2 = \{y\}$, $A_3 = R$ and  $B = L \backslash \{x,y\}$. 
  Then every edge of $G$ goes between $B$ and $A_3$ or between two of the sets $A_1,A_2,A_3$, as required. 
  
  Suppose now that $k \ge 2$, i.e. $\oddgirth(H) = 2k+1 \ge 5$. For $1 \leq i \leq k$, let $X_i$ be the set of vertices at distance $(i-1)$ from $x$ in $H'$, and let $Y_i$ be the set of vertices at distance $(i-1)$ from $y$ in $H'$.
  Note that $X_1 = \{x\}$ and $Y_1 = \{y\}$.
  Also, $X_i,Y_i$ lie in $L$ if $i$ is odd and in $R$ if $i$ is even.
  Write
\[
    L' := L \backslash \bigcup_{i=1}^k (X_i \cup Y_i), \quad 
    R' := R \backslash \bigcup_{i=1}^k (X_i \cup Y_i),
  \]
  We first claim that $\{X_1, \dots, X_k, Y_1, \dots, Y_k, L', R'\}$ forms a partition of $V(H)$. The sets $X_1,\dots,X_k$ are clearly pairwise-disjoint, and so are $Y_1,\dots,Y_k$. Also, all of these sets are disjoint from $L',R'$ by definition.
  So we only need to check $X_i$ and $Y_j$ are disjoint for every pair $1 \leq i,j \leq k$.
  Suppose for contradiction that there exists $u \in X_i \cap Y_j$ for some $1 \le i, j \le k$.
  Then $i \equiv j \pmod 2$, because otherwise $X_i,Y_j$ are contained in different parts of the bipartition $L,R$.
  By the definition of $X_i$ and $Y_j$, $H'$ has a path 
  $x = x_1, x_2, \dots, x_{i}=u$ and a path $y = y_1,y_2,\dots,y_{j}=u$.
  Then, $x=x_1, x_2,\dots,x_i=u=y_j,y_{j-1},\dots,y_1,y,x$ forms a closed walk of length $i+j-1$, which is odd as $i \equiv j \pmod 2$.
  Hence, $\oddgirth(H) \le 2k-1$, contradicting our assumption. 

 \begin{figure}
    \centering
    \begin{tikzpicture}[scale = 1]

    \coordinate (x1) at (0.5,1);
    \coordinate (y1) at (-0.5,1);
    \coordinate (x2) at (1,0);
    \coordinate (y2) at (-1,0);
    \coordinate (L') at (0,-1);
    \coordinate (R') at (0,-2);

    \draw (x1) node[fill=black,circle,minimum size=2pt,inner sep=0pt] {};
    \draw (y1) node[fill=black,circle,minimum size=2pt,inner sep=0pt] {};
    \draw (x2) ellipse (0.5 and 0.3);
    \draw (y2) ellipse (0.5 and 0.3);

    \draw (L') ellipse (0.5 and 0.3);
    \draw (R') ellipse (0.5 and 0.3);

    \draw (x1) node[above] {\footnotesize $x$};
    \draw (y1) node[above] {\footnotesize $y$};

    \draw (x2) node {\footnotesize $X_2$};
    \draw (y2) node {\footnotesize $Y_2$};
    \draw (L') node {\footnotesize $L'$};
    \draw (R') node {\footnotesize $R'$};

    \draw[thick] (x1) -- (y1);

    \coordinate (X21) at ($(x2) + ({0.5*cos(100)},{0.3*sin(100)})$);
    \draw[thick] (x1) -- (X21);

    \coordinate (Y21) at ($(y2) + ({0.5*cos(80)},{0.3*sin(80)})$);
    \draw[thick] (y1) -- (Y21);

    \coordinate (X22) at ($(x2) + ({0.5*cos(270)},{0.3*sin(270)})$);

    \coordinate (Y22) at ($(y2) + ({0.5*cos(270)},{0.3*sin(270)})$);

    \coordinate (L'1) at ($(L') + ({0.5*cos(60)},{0.3*sin(60)})$);
    \coordinate (L'2) at ($(L') + ({0.5*cos(120)},{0.3*sin(120)})$);

    \draw[thick] (X22) -- (L'1);
    \draw[thick] (Y22) -- (L'2);

    \coordinate (L'1) at ($(L') + ({0.5*cos(-90)},{0.3*sin(-90)})$);

    \coordinate (R'1) at ($(R') + ({0.5*cos(90)},{0.3*sin(90)})$);

    \draw[thick] (L'1) -- (R'1);
    \end{tikzpicture}
    \hspace{2cm}
    \begin{tikzpicture}[scale = 1]

    \coordinate (x1) at (0.5,1);
    \coordinate (y1) at (-0.5,1);
    \coordinate (x2) at (1,0);
    \coordinate (y2) at (-1,0);
    \coordinate (x3) at (1,-1);
    \coordinate (y3) at (-1,-1);
    \coordinate (L') at (0,-2);
    \coordinate (R') at (0,-3);

    \draw (x1) node[fill=black,circle,minimum size=2pt,inner sep=0pt] {};
    \draw (y1) node[fill=black,circle,minimum size=2pt,inner sep=0pt] {};
    \draw (x2) ellipse (0.5 and 0.3);
    \draw (y2) ellipse (0.5 and 0.3);
    \draw (x3) ellipse (0.5 and 0.3);
    \draw (y3) ellipse (0.5 and 0.3);

    \draw (L') ellipse (0.5 and 0.3);
    \draw (R') ellipse (0.5 and 0.3);

    \draw (x1) node[above] {\footnotesize $x$};
    \draw (y1) node[above] {\footnotesize $y$};

    \draw (x2) node {\footnotesize $X_2$};
    \draw (y2) node {\footnotesize $Y_2$};
    \draw (x3) node {\footnotesize $X_3$};
    \draw (y3) node {\footnotesize $Y_3$};
    \draw (L') node {\footnotesize $R'$};
    \draw (R') node {\footnotesize $L'$};

    \draw[thick] (x1) -- (y1);

    \coordinate (X21) at ($(x2) + ({0.5*cos(100)},{0.3*sin(100)})$);
    \draw[thick] (x1) -- (X21);

    \coordinate (Y21) at ($(y2) + ({0.5*cos(80)},{0.3*sin(80)})$);
    \draw[thick] (y1) -- (Y21);

    \coordinate (X22) at ($(x2) + ({0.5*cos(-90)},{0.3*sin(-90)})$);
    \coordinate (X31) at ($(x3) + ({0.5*cos(90)},{0.3*sin(90)})$);
    
    \draw[thick] (X22) -- (X31);

    \coordinate (Y23) at ($(y2) + ({0.5*cos(-90)},{0.3*sin(-90)})$);
    \coordinate (Y31) at ($(y3) + ({0.5*cos(90)},{0.3*sin(90)})$);
    
    \draw[thick] (Y22) -- (Y31);

    \coordinate (X32) at ($(x3) + ({0.5*cos(270)},{0.3*sin(270)})$);

    \coordinate (Y32) at ($(y3) + ({0.5*cos(270)},{0.3*sin(270)})$);

    \coordinate (L'1) at ($(L') + ({0.5*cos(60)},{0.3*sin(60)})$);
    \coordinate (L'2) at ($(L') + ({0.5*cos(120)},{0.3*sin(120)})$);

    \draw[thick] (X32) -- (L'1);
    \draw[thick] (Y32) -- (L'2);

    \coordinate (L'1) at ($(L') + ({0.5*cos(-90)},{0.3*sin(-90)})$);

    \coordinate (R'1) at ($(R') + ({0.5*cos(90)},{0.3*sin(90)})$);

    \draw[thick] (L'1) -- (R'1);
    \end{tikzpicture}
    \caption{Proof of \Cref{lem:chi3-decomposition}, $k=2$ (left) and $k=3$ (right). Edges indicate bipartite graphs where edges can be present.\label{fig:structure of H}}
  \end{figure}
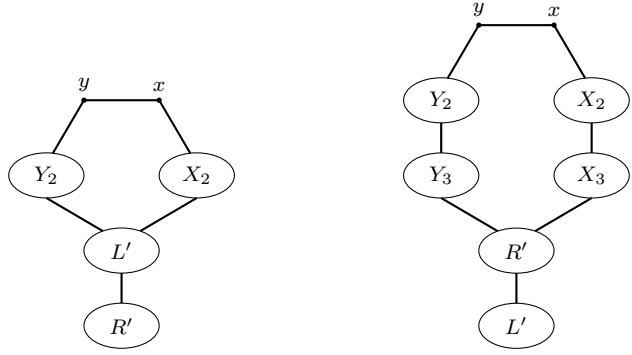

By definition, there are no edges between $X_i$ and $X_j$ for $j-i \geq 2$, and similarly for $Y_i,Y_j$. 
Also, there are no edges between $L' \cup R'$ and $\bigcup_{i=1}^{k-1}(X_i \cup Y_i)$ because the vertices in $L' \cup R'$ are at distance more than $k$ to $x,y$. Moreover, if $k$ is even then there are no edges between $X_k \cup Y_k$ and $R'$, and if $k$ is odd then there are no edges between $X_k \cup Y_k$ and $L'$. 
Next, we show that there are no edges between $X_i$ and $Y_j$ for any $1 \leq i,j \leq k$ except $(i,j) = (1,1)$.
Indeed, if $i=j$ then $e(X_i,Y_j) = 0$ because $X_i,Y_j$ are on the same side of the bipartition $L,R$. So suppose that $i \neq j$, say $i < j$, and assume by contradiction that there is an edge $uv$ with $u \in X_i, v \in Y_j$. Then $v$ is at distance at most $i+1 \leq k$ from $x$, implying that $Y_j$ intersects $X_1 \cup \dots \cup X_{i+1}$, a contradiction. 

Finally, we define the partition $A_1,\dots,A_{2k+1}$ that satisfies the assertion of the second item. 
If $k$ is even then take $A_1,\dots,A_{2k+1}$ to be 
$X_1,Y_1,\dots,Y_{k-1},Y_k \cup R',L',X_k,\dots,X_2$, and if $k$ is odd then take 
$A_1,\dots,A_{2k+1}$ to be 
$X_1,Y_1,\dots,Y_{k-1},Y_k \cup L',R',X_k,\dots,X_2$. 
See \Cref{fig:structure of H} for an illustration. By the above, in both cases it holds that $E(H) \subseteq \ecycle{2k+1}$, as required.
\end{proof}

For vertex $u \in V(G)$, denote by $N_G(u)$ the neighborhood of $u$ and let $\deg_G(u) = \abs{N_G(u)}$.
For vertices $u,v \in V(G)$, denote by $N_G(u,v)$ the common neighborhood of $u,v$ and let $\deg_G(u,v) = |N_G(u,v)|$.
\begin{lemma} \label{lem:ab-H-codegree}
  Let $H$ be a graph on $h$ vertices such that $\chi(H) = 3$ and $H$ contains a critical edge $xy$.
  Let $G$ be a graph on $n$ vertices with $\delta(G) \ge \alpha n$. 
  Let $ab \in E(G)$ such that $\deg_G(a,b) \ge \alpha n$.
  Then, there are at least $\poly(\alpha) n^{h-2}$ copies of $H$ in $G$ mapping $xy \in E(H)$ to $ab \in E(G)$.
\end{lemma}
\begin{proof}
  By the first item of \cref{lem:chi3-decomposition}, there is a partition $V(H) = A_1 \cupdot A_2 \cupdot A_3 \cupdot B$ such that $A_1=\{x\}, A_2=\{y\}$ and $E(H) \subseteq  (A_3 \times B) \cup \ecycle{3}$.
  Let $s = \abs{A_3}$ and $t = \abs{B}$.
  Each $u \in N_G(a,b)$ has at least $\alpha n - 2 \geq \frac{\alpha n}{2}$ neighbors not equal to $a,b$. 
  Hence, there are at least $\frac{1}{2} \cdot |N_G(a,b)| \cdot \frac{\alpha n}{2} \geq \frac{\alpha^2 n^2}{4}$ edges $uv$ with $u \in N_G(a,b)$ and $v \notin \{a,b\}$.
  Applying \cref{lem:blowup-count} with $H=K_2$, 
  $V_1 = N_G(a,b)$ and $V_2 = V(G) \backslash \{a,b\}$, we see that there are $\poly(\alpha) n^{s+t}$ pairs of disjoint sets $(S, T)$ such that $\abs{S} = s, \abs{T} = t$,
  $S \subseteq N_G(a,b)$, $a,b \notin T$, and $S,T$ form a complete bipartite graph in $G$.
  Given any such pair, it is safe to map $x$ to $a$, $y$ to $b$, $A_3$ to $S$ and $B$ to $T$ to obtain an $H$-copy.
  Hence, $G$ contains at least $\poly(\alpha) n^{s+t}=\poly(\alpha)n^{h-2}$ copies of $H$ mapping $xy$ to $ab$. 
\end{proof}

\begin{lemma} \label{lem:ab-H-general}
  Let $H$ be a graph on $h$ vertices such that $\chi(H) = 3$, $H$ contains a critical edge $xy$, and $\oddgirth(H) \ge 5$.
  Let $G$ be a graph on $n$ vertices, let $ab \in E(G)$, and suppose that there exists $A \subset N_G(a)$ and $B \subset N_G(b)$ such that $\abs{A}, \abs{B} \ge \alpha n$ and $\abs{N_{G}(a',b')} \ge \alpha n$ for all distinct $a' \in A$ and $b' \in B$. Then there are at least $\poly(\alpha) n^{h-2}$ copies of $H$ in $G$ mapping $xy \in E(H)$ to $ab \in E(G)$.
\end{lemma}
\begin{proof}
  By \cref{lem:chi3-decomposition} (using $\oddgirth(H) \ge 5$), there exists a partition $V(H) = A_1 \cupdot \dots \cupdot A_{5}$ such that ${A_1} = \{x\}, {A_2} = \{y\}$, and $E(H) \subseteq \ecycle{5}$.  Put $s_i=\abs{A_i}$ for $i \in [5]$.
  
  There are at least $(|A||B|- \abs{A})/2 \geq \alpha^2n^2/3$ pairs $\{a',b'\}$ of distinct vertices with $a' \in A, b' \in B$ (the factor of 2 is due to the fact that each pair in $A \cap B$ is counted twice). Each such pair $a',b'$ has at least $\alpha n - 2 \geq \alpha n/2$ common neighbors $c' \notin \{a,b\}$, by assumption.
  Therefore, there are at least $\frac{\alpha^2n^2}{3}\cdot \frac{\alpha n}{2}=\frac{\alpha^3n^3}{6}$ triples $(a',b',c')$ such that $a' \in A, b' \in B$, and $c'\neq a,b$ is a common neighbor of $a',b'$.
  By \cref{lem:blowup-count} with $H=K_{2,1}$ and $V_1=A,V_2=B,V_3=V(G)\backslash\{a,b\}$, 
  there are at least $\poly(\alpha) n^{s_3+s_4+s_5}$ corresponding copies of $K_{2,1}[s_3,s_5,s_4]$, i.e., triples of disjoint sets $(R,S,T)$ such that $R \subseteq A$, $S \subseteq B$, $a,b\notin T$, $\abs{R}=s_5, \abs{S}=s_3, \abs{T}=s_4$, and $(R,T)$ and $(S,T)$ form complete bipartite graphs in $G$.
  Given any such triple, we can safely map $A_1=\{x\}$ to $a$, $A_2=\{y\}$ to $b$, $A_5$ to $R$, $A_3$ to $S$, and $A_4$ to $T$ to obtain a copy of $H$.
  Thus, there are at least $\poly(\alpha) n^{s_3+s_4+s_5}=\poly(\alpha)n^{h-2}$ copies of $H$ mapping $xy$ to $ab$. 
  \end{proof}
\noindent
    In the following theorem we prove the upper bound in the second case of \Cref{thm:linear threshold}.
    \begin{theorem} \label{thm:3 chromatic with triangle}
  Let $H$ be a graph such that $\chi(H) = 3$, $H$ has a critical edge $xy$, and $H$ contains a triangle.
  Then, $\linearthres(H) \le \frac{1}{3}$.
\end{theorem}
\begin{proof}
  Write $h = v(H)$.
  Fix an arbitrary $\alpha > 0$, and
  let $G$ be an $n$-vertex graph with minimum degree $\delta(G) \ge (\frac{1}{3}+\alpha) n$ and with a collection $\mathcal{C} = \{H_1,\dots,H_m\}$ of $m := \varepsilon n^2$ edge-disjoint copies of $H$. 
  For each $i=1,\dots,m$, there exist $u,v,w \in V(H_i)$ forming a triangle (because $H$ contains a triangle).
  As $\deg_G(u)+\deg_G(v)+\deg_G(w) \ge 3\delta(G) \ge (1+3\alpha) n$, two of $u,v,w$ have at least $\alpha n$ common neighbors. 
  We denote these two vertices by $a_i$ and $b_i$.
  By \cref{lem:ab-H-codegree}, $G$ has at least $\poly(\alpha) n^{h-2}$ copies of $H$ which map $xy$ to $a_ib_i$. The edges $a_1b_1,\dots,a_mb_m$ are distinct because $H_1,\dots,H_m$ are edge-disjoint.
  Hence, summing over all $i = 1,\dots,m$, we see that $G$ contains at least $\varepsilon n^2 \cdot \poly(\alpha) n^{h-2} = \poly(\alpha) \varepsilon n^h$ 
  copies of $H$. This proves that  
  $\linearthres(H) \le \frac{1}{3} + \alpha$, and taking $\alpha \rightarrow 0$ gives $\linearthres(H) \le \frac{1}{3}$.
\end{proof}

In what follows, we need the following very well-known observation, originating in the work of Andr{\'a}sfai, Erd\H{o}s and S{\'o}s, see~\cite[Remark 1.6]{AES74}.
\begin{lemma}
\label{lem:shortest odd cycle}
    If $\delta(G) > \frac{2}{2k+1}n$ and $\oddgirth(G) \geq 2k+1$ for $k \ge 2$, then $G$ is bipartite.
\end{lemma}
\begin{proof}
Suppose by contradiction that $G$ is not bipartite and take a shortest odd cycle $C$ in $G$, so $|C| \geq 2k+1$. 
As $\sum_{x \in C}\deg(x) \geq (2k+1) \delta(G)> 2n$,
there exists a vertex $v \notin C$ with at least 3 neighbors on $C$. 
Then there are two neighbors $x,y \in C$ of $v$ such that the distance of $x,y$ along $C$ is not equal to $2$. Then by taking the odd path between $x,y$ along $C$ and adding the edges $vx,vy$, we get a shorter odd cycle, a contradiction. 
\end{proof}
\noindent
We will also use the following result of Letzter and Snyder, see~\cite[Corollary 32]{LS19}.
\begin{theorem}[\cite{LS19}]\label{thm:LS}
    Let $G$ be a $\{C_3,C_5\}$-free graph on $n$ vertices with $\delta(G) > \frac{n}{4}$. Then $G$ is homomorphic to $C_7$. 
\end{theorem}
\noindent
We can now finally prove the upper bound in the last case of \Cref{thm:linear threshold}. 

\begin{theorem}
  Let $H$ be a graph such that $\chi(H) = 3$, $H$ contains critical edge $xy$, and $\oddgirth(H) \ge 5$.
  Then $\linearthres(H) \le \frac{1}{4}$.
\end{theorem}
\begin{proof}
  Denote $h = |V(H)|$. Write $\oddgirth(G) = 2k+1 \ge 5$.
  By the second item of \cref{lem:chi3-decomposition}, there is a partition $V(H) = A_1 \cupdot A_2 \cupdot \dots \cupdot A_{2k+1}$ such that $\abs{A_1} = \abs{A_2} = 1$, and $E(H) \subseteq \ecycle{2k+1}$.
  Denote $s_i = \abs{A_i}$ for each $i \in [2k+1]$, 
  so $H$ is a subgraph of the blow-up $C_{2k+1}[s_1,\dots,s_{2k+1}]$ of $C_{2k+1}$.
  Let $c_1=c_1(C_{2k+1},s_1,\dots,s_{2k+1})>0$ and $c_2=c_2(k)>0$ be the constants given by \cref{lem:blowup-count} and \cref{lem:edge disjoint short cycle implies many Ck}, respectively.
  According to \cref{thm:poly threshold of odd cycle}, $\polythres(C_{2k+1}) = \frac{1}{2k+1} < \frac{1}{4}$, and hence there exists a constant $c_3=c_3(k)>0$ such that if $G$ is a graph on $n$ vertices with $\delta(G) \ge \frac{n}{4} $ and at least $\varepsilon n^2$ edge-disjoint $C_{2k+1}$-copies, then $G$ contains at least $c_3 \varepsilon^{\frac{1}{c_3}} n^{2k+1}$ copies of $C_{2k+1}$.
  Set $c := c_1 \cdot \min(c_2,c_3)$.
  
  Let $\alpha > 0$ and $\varepsilon$ be small enough; it suffices to assume that $\varepsilon < \left(\frac{\alpha^2}{200k(k+2)} \right)^{{1}/{c}}$.
  Let $G$ be a graph on $n$ vertices with $\delta(G) \ge (\frac{1}{4}+\alpha) n$ which contains at least $\varepsilon n^2$ edge-disjoint copies of $H$.
  Our goal is to show that $G$ contains $\Omega_{H,\alpha}(\varepsilon n^h)$ copies of $H$.
  Suppose first that $G$ contains at least $\varepsilon^{c} n^2$ edge-disjoint copies of $C_{2\ell+1}$ for some $1 \le \ell \le k$.
  If $\ell < k$, then $G$ contains $\Omega_k(\varepsilon^{c/c_2} n^{2k+1}) = \Omega_k(\varepsilon^{c_1} n^{2k+1})$ copies of $C_{2k+1}$ by \cref{lem:edge disjoint short cycle implies many Ck} and the choice of $c_2$.
  And if $\ell = k$, then $G$ contains $\Omega_k(\varepsilon^{{c}/{c_3}} n^{2k+1}) = \Omega_k(\varepsilon^{c_1} n^{2k+1})$ copies of $C_{2k+1}$ by \cref{thm:poly threshold of odd cycle} and the choice of $c_3$.
  In either case, $G$ contains $\Omega_k(\varepsilon^{c_1} n^{2k+1})$ copies of $C_{2k+1}$.
  But then, by \cref{lem:blowup-count} (with $V_1=\cdots=V_{2k+1}=V(G)$),
  $G$ contains at least 
  $\Omega_H(\varepsilon^{c_1/c_1} n^h)=\Omega_{H}(\varepsilon n^h)$ copies of $C_{2k+1}[s_1,\dots,s_{2k+1}]$, and hence $\Omega_{H,\alpha}(\varepsilon n^h)$ copies of $H$.
  This concludes the proof of this case.
  
  From now on, assume that $G$ contains at most $\varepsilon^c n^2$ edge-disjoint $C_{2\ell+1}$-copies for every $\ell \in [k]$.
  Let $\mathcal{C}_\ell$ be a maximal collection of edge-disjoint $C_{2\ell+1}$-copies in $G$, so $\abs{\mathcal{C}_\ell} \le \varepsilon^c n^2$.
  Let $E_c$ be the set of edges which are contained in one of the cycles in $\mathcal{C}_1 \cup \dots \cup \mathcal{C}_k$. Let $S$ be the set of vertices which are incident with at least $\frac{\alpha n}{10}$ edges from $E_c$. Then
 \begin{equation}\label{eq:S}
  \abs{E_c} \le \sum_{\ell=1}^k (2\ell+1) \varepsilon^c n^2 = k(k+2) \varepsilon^c n^2 \text{ and }
    \abs{S} \le 
    \frac{2\abs{E_c}}{\alpha n/10}
    \le 
    \frac{20 k (k+2) \varepsilon^c}{\alpha} n < \frac{\alpha n}{10},
 \end{equation}
  where the last inequality holds by our assumed bound on $\varepsilon$.
  Let $G'$ be the subgraph of $G$ obtained by deleting the edges in $E_c$ and the vertices in $S$. Note that $G' \subseteq G-E_c$ is $\{C_3,C_5,\dots,C_{2k+1}\}$-free because for every $1 \leq \ell \leq k$, we removed all edges from a maximal collection of edge-disjoint $C_{2\ell+1}$-copies.
  \begin{claim}\label{claim:properties of G'}
    $|V(G')| > (1-\frac{\alpha}{10}) n$ and 
    $\delta(G') > (\frac{1}{4}+\frac{4\alpha}{5}) n$.
  \end{claim}
  \begin{proof}
    The first inequality follows from \eqref{eq:S} as $|V(G')| = n-\abs{S}$.
    Each $v \in V(G) \backslash S$ has at most $\frac{\alpha n}{10}$ incident edges from $E_c$, and at most $\abs{S} < \frac{\alpha n}{10}$ neighbors in $S$, thereby $\deg_{G'}(v) > \deg_G(v) - \frac{\alpha n}{5} \ge (\frac{1}{4}+\frac{4\alpha}{5}) n$.
    Hence, $\delta(G') > (\frac{1}{4}+\frac{4\alpha}{5}) n$.
  \end{proof}
   \begin{claim} \label{claim:shape of G'}
    $G'$ is homomorphic to $C_7$. Moreover, $G'$ is bipartite unless $k=2$.
  \end{claim}
  \begin{proof}
  Recall that $G'$ is $\{C_3,C_5,\dots,C_{2k+1}\}$-free.
    As $k \geq 2$, $G'$ is $\{C_3,C_5\}$-free. Also, $\delta(G') > \frac{n}{4} \ge \frac{|V(G')|}{4}$ by \Cref{claim:properties of G'}.
    So $G'$ is homomorphic to $C_7$ by \Cref{thm:LS}.
    If $k \ge 3$, i.e. $\oddgirth(H) \ge 7$, then $\oddgirth(G') \ge 2k+3 \ge 9$. As $\delta(G') > \frac{n}{4}$, 
    $G'$ is bipartite by \Cref{lem:shortest odd cycle}.
     \end{proof}
  \noindent
  The rest of the proof is divided into two cases based whether or not $G'$ is bipartite. These cases are handled by Propositions~\ref{prop:bipartite} and~\ref{prop:not bipartite}, respectively.
  \begin{proposition}\label{prop:bipartite}
  Suppose that $G'$ is bipartite. Then $G$ has $\Omega_{H,\alpha}(\varepsilon n^h)$ copies of $H$.  
  \end{proposition}
  \begin{proof} 
    Let $(L', R')$ be a bipartition of $G'$, so $V(G) = L' \cupdot R' \cupdot S$.
    Let $L_1 \subseteq S$ (resp. $R_1 \subseteq S$) be the set of vertices of $S$ having at most $\frac{\alpha n}{5}$ neighbors in $L'$ (resp. $R'$).
    Let $G''$ be the bipartite subgraph of $G$ induced by the bipartition $(L'',R'') := (L' \cupdot L_1, R' \cupdot R_1)$.
    Let $S'' = V(G) \backslash (L'' \cupdot R'')$, so $V(G)=L''\cupdot R'' \cupdot S''$.

    We claim that $\delta(G'') \ge (\frac{1}{4} + \frac{\alpha}{2}) n$.
    First, as $G'$ is a subgraph of $G''$, we have $\deg_{G''}(v) > (\frac{1}{4}+\frac{4\alpha}{5}) n$
    for each $v \in V(G') \subseteq V(G'')$ by \Cref{claim:properties of G'}.
    Now we consider vertices in $V(G'') \setminus V(G') = L_1 \cup R_1$. 
    Each $v \in L_1$ has at most $\abs{S} \le \frac{\alpha n}{10}$ neighbors in $S$ and at most $\frac{\alpha n}{5}$ neighbors in $L'$, by the definition of $L_1$.
    Hence, $v$ has at least $\deg_G(v) - \frac{3 \alpha}{10}n \ge (\frac{1}{4} + \frac{\alpha}{2}) n$ neighbors in $R' \subseteq V(G'')$.
    By the symmetric argument for vertices $v \in R_1$, we get that $\delta(G'') \ge (\frac{1}{4} + \frac{\alpha}{2}) n$, as required. 

    For an edge $uv \in E(G) \backslash E(G'')$, we say $uv$ is of type I if $u, v \in L''$ or $u, v \in R''$, and we say that $uv$ is of type II if $u \in S''$ or $v \in S''$. Every edge in $E(G) \backslash E(G'')$ is of type I or II.
    Since $\chi(H) = 3$ and $G''$ is bipartite, each copy of $H$ in $G$ must contain an edge of type I or an edge of type II (or both).
    As $G$ has $\varepsilon n^2$ edge-disjoint $H$-copies, $G$ contains at least $\frac{\varepsilon n^2}{2}$ edges of type I or at least $\frac{\varepsilon n^2}{2}$ edges of type II.
    We now consider these two cases separately. 
    See \cref{fig:bipartite proof} for an illustration.
    Recall that $xy \in E(H)$ denotes a critical edge of $H$. 
    \paragraph{Case 1:} {\em $G$ contains $\frac{\varepsilon n^2}{2}$ edges of type I. }
        Fix any edge $ab \in E(G)$ of type I. Without loss of generality, assume $a, b \in L''$ (the case $a,b \in R''$ is symmetric). We claim that 
        $G$ has $\poly(\alpha) n^{h-2}$ copies of $H$ mapping $xy \in E(H)$ to $ab \in E(G)$.
        If $\deg_G(a,b) \ge \frac{\alpha n}{2}$ then this holds by \cref{lem:ab-H-codegree}.
        Otherwise, $\deg_G(a,b) < \frac{\alpha n}{2}$, and thus 
        $$\abs{R''} \ge \abs{N_{G''}(a) \cup N_{G''}(b)} \ge \deg_{G''}(a) + \deg_{G''}(b) - \deg_G(a,b) > 2\delta(G'') - \frac{\alpha n}{2} 
        > \frac{n}{2},$$
        using that $\delta(G'') \geq (\frac{1}{4} + \frac{\alpha}{2})n$.
        Thus, $\abs{L''} < \frac{n}{2}$. 
        This implies that for all $a' \in N_{G''}(a), b' \in N_{G''}(b)$,
        $$\deg_{G''}(a',b') \ge 2\delta(G'') - \abs{L''} \ge \alpha n.$$
        Now, by \cref{lem:ab-H-general} (with $A = N_{G''}(a)$ and $B = N_{G''}(b)$), there are  $\poly(\alpha) n^{h-2}$ copies of $H$ mapping $xy$ to $ab$, as claimed. 
        Summing over all edges $ab$ of type I, we get $\frac{\varepsilon n^2}{2} \cdot \poly(\alpha) n^{h-2} = \poly(\alpha) \varepsilon n^{h}$ copies of $H$.
        This completes the proof in {Case 1}.

        \paragraph{Case 2:} {\em $G$ contains $\frac{\varepsilon n^2}{2}$ edges of type II. }
        Note that the number of edges of type II is trivially at most $\abs{S''} n$.
        Thus, $\abs{S''} \ge \frac{\varepsilon n}{2}$.
        Fix some $a \in S''$.
        By the definition of $L'', R''$ and $S''$, $v$ has at least $\frac{\alpha n}{5}$ neighbors in $L' \subseteq L''$ and at least $\frac{\alpha n}{5}$ neighbors in 
        $R' \subseteq R''$.
        Without loss of generality, assume $\abs{L''} \le \abs{R''}$, thereby $\abs{L''} \le \frac{n}{2}$.
        Now fix any $b \in L''$ adjacent to $a$; there are at least $\frac{\alpha n}{5}$ choices for $b$.  
        We have $\abs{N_{G}(a) \cap R''} \ge \frac{\alpha n}{5}$ and $\abs{N_{G''}(b)} \ge \delta(G'') > \frac{n}{4}$, and for all 
        $a' \in N_{G}(a) \cap R'', b' \in N_{G''}(b) \subseteq R''$ it holds that $\deg_{G''}(a',b') \ge 2\delta(G'') - \abs{L''} \ge \alpha n$.
       Therefore, by \cref{lem:ab-H-general}, $G$ has $\poly(\alpha)n^{h-2}$ copies of $H$ mapping $xy$ to $ab$. 
        Enumerating over all $a \in S''$ and $b \in N_G(a) \cap L''$,
        we again get $\Omega_{H,\alpha}(\varepsilon n^h)$ copies of $H$ in $G$. This completes the proof of \Cref{prop:bipartite}.  
    \end{proof} 
        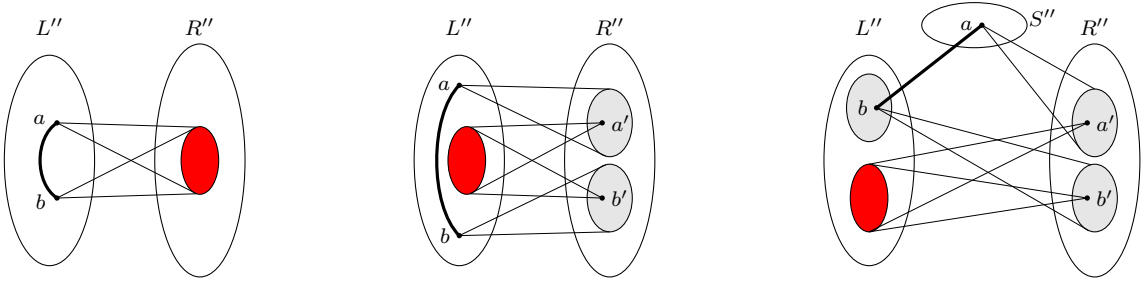
\begin{figure}
          \centering
          \begin{tikzpicture}[scale = 1]
            \coordinate (L) at (0,0);
            \coordinate (R) at (2,0);
            \draw (L) ellipse (0.6 and 1.4);
            \draw (R) ellipse (0.6 and 1.55);
            
            \coordinate (LL) at ($(L) + (0, 1.55)$);
            \coordinate (RR) at ($(R) + (0, 1.55)$);
            \draw (LL) node[above] {\footnotesize $L''$};
            \draw (RR) node[above] {\footnotesize $R''$};
            
            \coordinate (a) at ($(L) + ({0.1},{0.5})$);
            \coordinate (b) at ($(L) + ({0.1},{-0.5})$);
            \draw[black, very thick] (a)..controls(-0.2,0.3)and(-0.2,-0.3) ..(b);
            \draw (a) node[fill=black,circle,minimum size=2pt,inner sep=0pt] {};
            \draw (b) node[fill=black,circle,minimum size=2pt,inner sep=0pt] {};
            \coordinate (aa) at ($(a) + ({-0.03},{0.05})$);
            \coordinate (bb) at ($(b) + ({-0.03},{-0.05})$);
            \draw (aa) node[left] {\footnotesize $a$};
            \draw (bb) node[left] {\footnotesize $b$};
            
            \coordinate (C) at ($(R) + ({0, 0})$);
            \draw [black,fill=red] (C) ellipse (0.25 and 0.45);
            \coordinate (C1) at ($(C) + ({0.25*cos(100)}, {0.45*sin(100)})$);
            \coordinate (C2) at ($(C) + ({0.25*cos(260)}, {0.45*sin(260)})$);
            
            \draw[very thin] (a) -- (C1);
            \draw[very thin] (b) -- (C1);
            \draw[very thin] (a) -- (C2);
            \draw[very thin] (b) -- (C2);
          \end{tikzpicture}
          \hspace{2cm}
          \begin{tikzpicture}[scale = 1]
            \coordinate (L) at (0,0);
            \coordinate (R) at (2,0);
            \draw (L) ellipse (0.6 and 1.4);
            \draw (R) ellipse (0.6 and 1.55);
            
            \coordinate (LL) at ($(L) + (0, 1.55)$);
            \coordinate (RR) at ($(R) + (0, 1.55)$);
            \draw (LL) node[above] {\footnotesize $L''$};
            \draw (RR) node[above] {\footnotesize $R''$};
            
            \coordinate (a) at ($(L) + ({0},{1})$);
            \coordinate (b) at ($(L) + ({0},{-1})$);
            \draw[black, very thick] (a)..controls(-0.4,0.6)and(-0.4,-0.6) ..(b);
            \draw (a) node[fill=black,circle,minimum size=2pt,inner sep=0pt] {};
            \draw (b) node[fill=black,circle,minimum size=2pt,inner sep=0pt] {};
            \draw (a) node[left] {\footnotesize $a$};
            \draw (b) node[left] {\footnotesize $b$};
            
            \coordinate (A1) at ($(R) + ({0, 0.5})$);
            \draw [black,fill=gray!20] (A1) ellipse (0.3 and 0.45);
            \coordinate (A11) at ($(A1) + ({0, 0.45})$);
            \coordinate (A12) at ($(A1) + ({0.3*cos(260)}, {0.45*sin(260)})$);
            \coordinate (a1) at ($(A1) + ({-0.1, 0})$);
            \coordinate (B1) at ($(R) + ({0, -0.5})$);
            \draw [black,fill=gray!20] (B1) ellipse (0.3 and 0.45);
            \coordinate (B11) at ($(B1) + ({0.3*cos(100)}, {0.45*sin(100)})$);
            \coordinate (B12) at ($(B1) + ({0, -0.45})$);
            \coordinate (b1) at ($(B1) + ({-0.1, 0})$);
            
            \draw (a1) node[fill=black,circle,minimum size=2pt,inner sep=0pt] {};
            \draw (b1) node[fill=black,circle,minimum size=2pt,inner sep=0pt] {};
            \draw (a1) node[right] {\footnotesize $a'$};
            \draw (b1) node[right] {\footnotesize $b'$};
            
            \draw[very thin] (a) -- (A11);
            \draw[very thin] (a) -- (A12);
            \draw[very thin] (b) -- (B11);
            \draw[very thin] (b) -- (B12);
            
            \coordinate (C) at ($(L) + ({0.1, 0})$);
            \draw [black,fill=red] (C) ellipse (0.25 and 0.45);
            \coordinate (C1) at ($(C) + ({0.25*cos(80)}, {0.45*sin(80)})$);
            \coordinate (C2) at ($(C) + ({0.25*cos(280)}, {0.45*sin(280)})$);
            \draw[very thin] (a1) -- (C1);
            \draw[very thin] (a1) -- (C2);
            \draw[very thin] (b1) -- (C1);
            \draw[very thin] (b1) -- (C2);
          \end{tikzpicture}
          \hspace{2cm}
          \begin{tikzpicture}[scale=1]
            \coordinate (L) at (0,0);
            \coordinate (R) at (3,0);
            \draw (L) ellipse (0.6 and 1.4);
            \draw (R) ellipse (0.6 and 1.55);
            
            \coordinate (LL) at ($(L) + (0, 1.55)$);
            \coordinate (RR) at ($(R) + (0, 1.55)$);
            \draw (LL) node[above] {\footnotesize $L''$};
            \draw (RR) node[above] {\footnotesize $R''$};
            
            \coordinate (S) at (1.4,1.8);
            \draw (S) ellipse (0.7 and 0.3);
            \coordinate (SS) at ($(S) + (0.6, 0.1)$);
            \draw (SS) node[right] {\footnotesize $S''$};
            
            \coordinate (a) at ($(S) + ({0.1}, {0})$);
            \draw (a) node[fill=black,circle,minimum size=2pt,inner sep=0pt] {};
            \draw (a) node[left] {\footnotesize $a$};
            \coordinate (N) at ($(L) + ({0},{0.7})$);
            \draw [black,fill=gray!20] (N) ellipse (0.3 and 0.45);
            \coordinate (b) at ($(N) + ({0.1}, {0})$);
            \draw (b) node[fill=black,circle,minimum size=2pt,inner sep=0pt] {};
            \draw (b) node[left] {\footnotesize $b$};
            \draw[very thick] (a) -- (b);
            
            \coordinate (A1) at ($(R) + ({0, 0.5})$);
            \draw [black,fill=gray!20] (A1) ellipse (0.3 and 0.45);
            \coordinate (A11) at ($(A1) + ({0, 0.45})$);
            \coordinate (A12) at ($(A1) + ({0.3*cos(230)}, {0.45*sin(230)})$);
            \coordinate (a1) at ($(A1) + ({-0.1, 0})$);
            \coordinate (B1) at ($(R) + ({0, -0.5})$);
            \draw [black,fill=gray!20] (B1) ellipse (0.3 and 0.45);
            \coordinate (B11) at ($(B1) + ({0.3*cos(100)}, {0.45*sin(100)})$);
            \coordinate (B12) at ($(B1) + ({0.3*cos(255)},{ 0.45*sin(255)})$);
            \coordinate (b1) at ($(B1) + ({-0.1, 0})$);
            
            \draw (a1) node[fill=black,circle,minimum size=2pt,inner sep=0pt] {};
            \draw (b1) node[fill=black,circle,minimum size=2pt,inner sep=0pt] {};
            \draw (a1) node[right] {\footnotesize $a'$};
            \draw (b1) node[right] {\footnotesize $b'$};
            
            \draw[very thin] (a) -- (A11);
            \draw[very thin] (a) -- (A12);
            \draw[very thin] (b) -- (B11);
            \draw[very thin] (b) -- (B12);
            
            \coordinate (C) at ($(L) + ({0, -0.5})$);
            \draw [black,fill=red] (C) ellipse (0.25 and 0.45);
            \coordinate (C1) at ($(C) + ({0, 0.45})$);
            \coordinate (C2) at ($(C) + ({0, -0.45})$);
            \draw[very thin] (a1) -- (C1);
            \draw[very thin] (a1) -- (C2);
            \draw[very thin] (b1) -- (C1);
            \draw[very thin] (b1) -- (C2);
          \end{tikzpicture}
          \caption{Proof of \cref{prop:bipartite}: Case 1 with $\deg_G(a,b) \ge \frac{\alpha n}{2}$ (left), Case 1 with $\deg_G(a,b) < \frac{\alpha n}{2}$ (middle) and Case 2 (right). The red part is the common neighborhood of $a$ and $b$ (or $a'$ and $b'$).\label{fig:bipartite proof}}
        \end{figure}

    \begin{proposition}\label{prop:not bipartite}
  Suppose $G'$ is non-bipartite but homomorphic to $C_7$. Then $G$ has $\Omega_{H,\alpha}(\varepsilon n^h)$ copies \nolinebreak of \nolinebreak $H$.  
  \end{proposition}
    \begin{proof}
   By \cref{claim:shape of G'} we must have $k=2$ , so $\oddgirth(H) = 5$.
    The proof is similar to that of \Cref{prop:bipartite}, but instead of a bipartition of $G'$, we use a partition corresponding to a homomorphism into $C_7$. 
     Let $V(G) \backslash S = V(G') = V_1' \cupdot V_2' \cupdot \dots \cupdot V_7'$ be a partition of $V(G')$ such that $E(G') \subseteq \bigcup_{i \in [7]} V_i' \times V_{i+1}'$.
     Here and later, all subscripts are modulo 7.
     We have $V_i' \neq \emptyset$ for all $i \in [7]$, because otherwise $G'$ would be bipartite. 
    For $i \in [7]$, let $S_i$ be the set of vertices in $S$ having at most $\frac{2\alpha n}{5}$ neighbors in 
    $V(G') \setminus (V'_{i-1} \cup V'_{i+1})$.
    In case $v$ lies in multiple $S_i$'s, we put $v$ arbitrarily in one of them.
    Set $V_i'' := V_i' \cup S_i$.
    Let $G''$ be the 7-partite subgraph of $G$ with parts $V''_1,\dots,V''_7$ and with all edges of $G$ between $V_i''$ and $V_{i+1}''$, $i = 1,\dots,7$.
    By definition, $G'$ is a subgraph of $G''$, and $G''$ is homomorphic to $C_7$ via the homomorphism $V''_i \mapsto i$. Put $S'' := V(G) \backslash V(G'') = S \setminus \bigcup_{i=1}^7 S_i$.
    We now collect the following useful properties.
    \begin{claim}\label{claim:G'' case 2}
    The following holds:
    \begin{itemize}
    \item[(i)] {\em $\delta(G'') \ge (\frac{1}{4} + \frac{\alpha}{2}) n$.} 
      \item[(ii)] For every $i \in [7]$ and for every $u,v \in V''_i$ or $u \in V''_i, v \in V''_{i+2}$, it holds that $\deg_{G''}(u,v) \geq \alpha n$.
      \item[(iii)] For every $i \in [7]$, every $v \in V_i''$ has at least $\alpha n$ neighbors in $V_{i-1}''$ and at least $\alpha n$ neighbors in $V_{i+1}''$. 
      \item[(iv)]
      For every $a \in S''$, there are $i, j$ with $j-i \equiv 1,3 \pmod{7}$ and $\abs{N_G(a)\cap V_i''}, \abs{N_G(a)\cap V_j''} > \frac{2\alpha n}{25}$.
    \end{itemize}
    \end{claim}
    \noindent
    \begin{proof}
    {\color{white} fds}
        \begin{itemize}
        \item[(i)]
        Let $i \in [7]$ and $v \in V_i''$. 
        If $v \in V(G')$, then $\deg_{G''}(v) \ge \deg_{G'}(v) \ge \delta(G') > (\frac{1}{4} + \frac{\alpha}{2}) n$, using \Cref{claim:properties of G'}.
        Otherwise, $v \in S_i$.
        By definition, $v$ has at most $\frac{2\alpha n}{5}$ neighbours in $V(G') \setminus (V'_{i-1} \cup V'_{i+1})$. Also, $v$ has at most $\abs{S} \leq \frac{\alpha n}{10}$ neighbours in $S$. It follows that 
        $v$ has at least $\deg_G(v) - \frac{2\alpha n}{5} - \frac{\alpha n}{10} \ge (\frac{1}{4} + \frac{\alpha}{2}) n$ neighbors in $V_{i-1}'' \cup V_{i+1}''$. Hence, $\deg_{G''}(v) > (\frac{1}{4}+\frac{\alpha}{2}) n$.
        \item[(ii)]
        First, observe that 
        \begin{equation}\label{eq:i,i+2}
        \abs{V_{i}''} + \abs{V_{i+2}''} \ge 
        \left(\frac{1}{4}+\frac{\alpha}{2}\right) n
        \end{equation}
        for all $i \in [7]$. Indeed, $V''_{i+1}$ is non-empty, and fixing any $v \in V''_{i+1}$, we have 
        $\abs{V_{i}''} + \abs{V_{i+2}''} \geq \deg_{G''}(v) \geq \delta(G'') \geq (\frac{1}{4}+\frac{\alpha}{2}) n$. By applying \eqref{eq:i,i+2} to the pairs $(i+2,i+4)$ and $(i-2,i)$, we get
        \begin{equation}\label{eq:i-1,i+1,i+3}
        \abs{V''_{i-1}} + \abs{V''_{i+1}} + \abs{V''_{i+3}} \leq n - (\abs{V_{i+2}''} + \abs{V_{i+4}''}) - 
        (\abs{V_{i-2}''} + \abs{V_{i}''}) \leq n - 2\left(\frac{1}{4}+\frac{\alpha}{2}\right) n < \frac{n}{2}.
        \end{equation}
        Now let $i \in [7]$. For $u,v\in V''_i$ we have 
        $N_{G''}(u)\cup N_{G''}(v) \subseteq V_{i-1}'' \cup V_{i+1}''$, and for $u \in V''_i, v \in V''_{i+2}$ we have
        $N_{G''}(u)\cup N_{G''}(v) \subseteq V_{i-1}'' \cup V_{i+1}'' \cup V''_{i+3}$.
        In both cases, 
        $\abs{N_{G''}(u)\cup N_{G''}(v)} < \frac{n}{2}$ by \eqref{eq:i-1,i+1,i+3}.
        As
        $\deg_{G''}(u)+\deg_{G''}(v) \ge 2\delta(G'') \ge (\frac{1}{2} + \alpha) n$, we have $\deg_{G''}(u,v) > \alpha n$, as required.
        \item[(iii)] 
        We first argue that $|V''_i| \leq (\frac{1}{4} - \frac{3\alpha}{2})n$ for each $i \in [7]$. Indeed, by applying \eqref{eq:i,i+2} to the pairs $(i-1,i+1)$, $(i+2,i+4)$, $(i+3,i+5)$, we get
        $$
        \abs{V''_i} \leq n - 
        (\abs{V_{i-1}''} + \abs{V_{i+1}''}) - 
        (\abs{V_{i+2}''} + \abs{V_{i+4}''}) - 
        (\abs{V_{i+3}''} + \abs{V_{i+5}''}) \leq 
        n - 3\left(\frac{1}{4}+\frac{\alpha}{2}\right) n = \left(\frac{1}{4} - \frac{3\alpha}{2}\right)n.  
        $$
        Now, for every $v \in V''_i$, we have $N_{G''}(v) \subseteq V''_{i-1} \cup V''_{i+1}$ and $\abs{V''_{i-1}},\abs{V''_{i+1}} < (\frac{1}{4} - \frac{3\alpha}{2})n$. Hence, $v$ has at least $\deg_{G''}(v) - (\frac{1}{4} - \frac{3\alpha}{2})n \geq \alpha n$ neighbors in each of $V''_{i-1},V''_{i+1}$. 
        \item[(iv)]
        Let $I$ be the set of $i$ with $\abs{N_G(a) \cap V_i''} \geq \frac{2\alpha n}{25}$. If $I$ is empty, then $a$ has less than  $5\cdot \frac{2\alpha n}{25} =\frac{2\alpha n}{5}$ neighbors in 
    every $V(G') \setminus (V'_{i-1} \cup V'_{i+1})$ and therefore can not be in $S''$. Suppose for contradiction that there exist no $i,j \in I$ with $j-i \equiv 1,3 \pmod{7}$. 
    We claim that there is $j \in [7]$ such that $I \subseteq \{j,j+2\}$. 
    Fix an arbitrary $i \in I$. 
    Then, $i\pm 1, i \pm 3 \notin I$ by assumption.
    Also, at most one of $i+2,i-2$ is in $I$, because $(i-2) - (i+2) \equiv 3 \pmod{7}$.
    So $I \subseteq \{i,i+2\}$ or $I \subseteq \{i-2,i\}$, proving our claim that $I \subseteq \{j,j+2\}$ for some $j$. 
    By the definition of $I$, $a$ has at most $5 \cdot \frac{2\alpha n}{25}=\frac{2\alpha n}{5}$ neighbors in $V(G') \backslash (V_j' \cup V_{j+2}')$. Hence, $a \in S_{j+1}$.
        This contradicts the fact that $a \in S''$, as $S'' \cap S_{i+1} = \emptyset$. 
         \end{itemize}
    \end{proof}
      
    We continue with the proof of \Cref{prop:not bipartite}.
    Recall that the edges in $E(G) \setminus E(G'')$ are precisely the edges of $G$ not belonging to ${\bigcup_{i\in[7]} V''_i \times V''_{i+1}}$. 
    For an edge $ab \in E(G)\backslash E(G'')$, we say $ab$ is of type I if $a,b \in V(G'')$, and of type II if $a \in S''$ or $b \in S''$.
    Clearly, every edge in $E(G)\backslash E(G'')$ is either of type I or of type II. 
    Since $\oddgirth(H) = 5$ and $C_5$ is not homomorphic to $C_7$, every $H$-copy in $G$ must contain some edge of type I or of type II (or both).
    As $G$ has $\varepsilon n^2$ edge-disjoint $H$-copies, $G$ must have at least $\frac{\varepsilon n^2}{2}$ edges of type I or at least $\frac{\varepsilon n^2}{2}$ edges of type II.
    We consider these two cases separately. 
    See \cref{fig:non-bipartite proof} for an illustration.
    Recall that $xy \in E(H)$ denotes a critical edge of $H$. 
      \paragraph{Case 1:} {\em $G$ contains $\frac{\varepsilon n^2}{2}$ edges of type I. }
        Fix any edge $ab$ of type I, where $a \in V_i''$ and $b \in V_j''$ for $i,j \in [7]$. 
        We now show that $G$ has $\poly(\alpha) n^{h-2}$ copies of $H$ mapping $xy \in E(H)$ to $ab$.
        As $ab \notin E(G'')$, we have $i-j \equiv 0,\pm 2,\pm 3 \pmod{7}$.
        When $j-i \equiv 0,\pm 2 \pmod{7}$, we have $\deg_G(a,b) \ge \deg_{G''}(a,b) > \alpha n$ by \Cref{claim:G'' case 2} (ii).
        Then, by \cref{lem:ab-H-codegree}, $G$ has $\poly(\alpha) n^{h-2}$ copies of $H$ mapping $xy$ to $ab$, as required. 
        Now suppose that $j-i \equiv \pm 3 \pmod{7}$, say $j \equiv i+3 \pmod{7}$.
        Denote $A := N_{G}(a) \cap V_{i-1}''$ and $B := N_G(b) \cap V_{j+1}''=N_G(b)\cap V_{i-3}''$.
        We have that $\abs{A}, \abs{B} \ge \alpha n$ by \Cref{claim:G'' case 2} (iii), and $\abs{N_G(a',b')} > \alpha n$ for all $a' \in A, b' \in B$ by \Cref{claim:G'' case 2} (ii).
        Now, by \cref{lem:ab-H-general}, $G$ has $\poly(\alpha) n^{h-2}$ copies of $H$ mapping $xy$ to $ab$, proving our claim.
        Summing over all edges $ab$ of type I, we get $\frac{\varepsilon n^2}{2} \cdot \poly(\alpha) n^{h-2}=\Omega_{H,\alpha}(\varepsilon n^h)$ copies of $H$ in $G$, finishing this case.

      \paragraph{Case 2:} {\em $G$ contains $\frac{\varepsilon n^2}{2}$ edges of type II.}
        Notice that the number edges incident to $S''$ is at most $\abs{S''} n$, meaning that $\abs{S''} \ge \frac{\varepsilon n}{2}$.
        Fix any $a \in S''$.
        By \Cref{claim:G'' case 2} (iv), there exist $i, j \in [7]$ with $j-i \equiv 1,3 \pmod{7}$ and $\abs{N_G(a)\cap V_i''}, \abs{N_G(a)\cap V_j''} > \frac{2\alpha n}{25}$.
        Fix any $b \in N_G(a) \cap V_i''$ (there are at least $\frac{2\alpha n}{25}$ choices for $b$). Take $A = N_G(a) \cap V_{j}''$ and $B = N_G(b) \cap V_{i+1}''$.
        We have that $\abs{A} \ge \frac{2\alpha n}{25}$, and $\abs{B} \ge \alpha n$ by \Cref{claim:G'' case 2} (iii).
        Further, as $j - (i+1) \equiv 0,2 \pmod{7}$, \Cref{claim:G'' case 2} (ii) implies that $\abs{N_G(a',b')} > \alpha n$ for all $a' \in A, b' \in B$.
        Now, by \cref{lem:ab-H-general}, $G$ has $\poly(\alpha) n^{h-2}$ copies of $H$ mapping $xy$ to $ab$.
        Summing over all choices of $a \in S''$ and $b \in V_{i}''$, we acquire $\abs{S''}\cdot \frac{2\alpha n}{25} \cdot \poly(\alpha) n^{h-2}=\Omega_{H,\alpha}(\varepsilon n^h)$ copies of $H$ in $G$. This completes the proof of Case 2, and hence the proposition. 
    \hfill
    \end{proof}
    \begin{figure}
          \centering
          \begin{tikzpicture}[scale = 0.9]
            \coordinate (V1) at (-0.85,0);
            \coordinate (V2) at (0.85,0);
            \coordinate (V7) at (-1.6,-1);
            \coordinate (V3) at (1.6,-1);
            \coordinate (V6) at (-1.45,-2);
            \coordinate (V4) at (1.45,-2);
            \coordinate (V5) at (0,-2.75);
            \draw (V1) ellipse (0.6 and 0.3);
            \draw (V2) ellipse (0.6 and 0.3);
            \draw (V3) ellipse (0.6 and 0.3);
            \draw (V4) ellipse (0.6 and 0.3);
            \draw (V5) ellipse (0.6 and 0.3);
            \draw (V6) ellipse (0.6 and 0.3);
            \draw (V7) ellipse (0.6 and 0.3);
            \coordinate (VV1) at ($(V1) + (-0.85, 0)$);
            \coordinate (VV2) at ($(V2) + (0.85, 0)$);
            \coordinate (VV3) at ($(V3) + (0.85, 0)$);
            \coordinate (VV4) at ($(V4) + (0.85, 0)$);
            \coordinate (VV5) at ($(V5) + (0.85, 0)$);
            \coordinate (VV6) at ($(V6) + (-0.85, 0)$);
            \coordinate (VV7) at ($(V7) + (-0.85, 0)$);
            \draw (VV1) node {\footnotesize $V_1''$};
            \draw (VV2) node {\footnotesize $V_2''$};
            \draw (VV3) node {\footnotesize $V_3''$};
            \draw (VV4) node {\footnotesize $V_4''$};
            \draw (VV5) node {\footnotesize $V_5''$};
            \draw (VV6) node {\footnotesize $V_6''$};
            \draw (VV7) node {\footnotesize $V_7''$};
            
            \coordinate (a) at ($(V1) + ({0},{0})$);
            \coordinate (b) at ($(V3) + ({0},{0})$);
            \draw[black, very thick] (a) -- (b);
            \draw (a) node[fill=black,circle,minimum size=2pt,inner sep=0pt] {};
            \draw (b) node[fill=black,circle,minimum size=2pt,inner sep=0pt] {};
            \draw (a) node[left] {\footnotesize $a$};
            \draw (b) node[right] {\footnotesize $b$};
            
            \coordinate (C) at ($(V2) + ({0, 0})$);
            \draw [black,fill=red] (C) ellipse (0.35 and 0.15);
            \coordinate (Ca1) at ($(C) + ({0.35*cos(100)}, {0.15*sin(100)})$);
            \coordinate (Ca2) at ($(C) + ({0.35*cos(260)}, {0.15*sin(260)})$);
            \coordinate (Cb1) at ($(C) + ({0.35*cos(200)}, {0.15*sin(200)})$);
            \coordinate (Cb2) at ($(C) + ({0.35*cos(10)}, {0.15*sin(10)})$);
            
            \draw[very thin] (a) -- (Ca1);
            \draw[very thin] (a) -- (Ca2);
            \draw[very thin] (b) -- (Cb1);
            \draw[very thin] (b) -- (Cb2);
          \end{tikzpicture}
          \hspace{0.4cm}
          \begin{tikzpicture}[scale = 0.9]
            \coordinate (V1) at (-0.85,0);
            \coordinate (V2) at (0.85,0);
            \coordinate (V7) at (-1.6,-1);
            \coordinate (V3) at (1.6,-1);
            \coordinate (V6) at (-1.45,-2);
            \coordinate (V4) at (1.45,-2);
            \coordinate (V5) at (0,-2.75);
            \draw (V1) ellipse (0.6 and 0.3);
            \draw (V2) ellipse (0.6 and 0.3);
            \draw (V3) ellipse (0.6 and 0.3);
            \draw (V4) ellipse (0.6 and 0.3);
            \draw (V5) ellipse (0.6 and 0.3);
            \draw (V6) ellipse (0.6 and 0.3);
            \draw (V7) ellipse (0.6 and 0.3);
            \coordinate (VV1) at ($(V1) + (-0.85, 0)$);
            \coordinate (VV2) at ($(V2) + (0.85, 0)$);
            \coordinate (VV3) at ($(V3) + (0.85, 0)$);
            \coordinate (VV4) at ($(V4) + (0.85, 0)$);
            \coordinate (VV5) at ($(V5) + (0.85, 0)$);
            \coordinate (VV6) at ($(V6) + (-0.85, 0)$);
            \coordinate (VV7) at ($(V7) + (-0.85, 0)$);
            \draw (VV1) node {\footnotesize $V_1''$};
            \draw (VV2) node {\footnotesize $V_2''$};
            \draw (VV3) node {\footnotesize $V_3''$};
            \draw (VV4) node {\footnotesize $V_4''$};
            \draw (VV5) node {\footnotesize $V_5''$};
            \draw (VV6) node {\footnotesize $V_6''$};
            \draw (VV7) node {\footnotesize $V_7''$};
            
            \coordinate (a) at ($(V1) + ({0},{0})$);
            \coordinate (b) at ($(V4) + ({0},{0})$);
            \draw[black, very thick] (a) -- (b);
            \draw (a) node[fill=black,circle,minimum size=2pt,inner sep=0pt] {};
            \draw (b) node[fill=black,circle,minimum size=2pt,inner sep=0pt] {};
            \draw (a) node[left] {\footnotesize $a$};
            \draw (b) node[right] {\footnotesize $b$};
            
            \coordinate (A1) at ($(V7) + ({0, 0})$);
            \draw [black,fill=gray!20] (A1) ellipse (0.4 and 0.23);
            \coordinate (A11) at ($(A1) + ({0.4*cos(145)}, {0.23*sin(145)})$);
            \coordinate (A12) at ($(A1) + ({0.4*cos(350)}, {0.23*sin(350)})$);
            \coordinate (a1) at ($(A1) + ({-0.1, 0})$);
            \coordinate (B1) at ($(V5) + ({0, 0})$);
            \draw [black,fill=gray!20] (B1) ellipse (0.4 and 0.23);
            \coordinate (B11) at ($(B1) + ({0.4*cos(110)}, {0.23*sin(110)})$);
            \coordinate (B12) at ($(B1) + ({0.4*cos(320)}, {0.23*sin(320)})$);
            \coordinate (b1) at ($(B1) + ({-0.1, 0})$);
            
            \draw (a1) node[fill=black,circle,minimum size=2pt,inner sep=0pt] {};
            \draw (b1) node[fill=black,circle,minimum size=2pt,inner sep=0pt] {};
            \draw (a1) node[right] {\footnotesize $a'$};
            \draw (b1) node[right] {\footnotesize $b'$};
            \draw[very thin] (a) -- (A11);
            \draw[very thin] (a) -- (A12);
            \draw[very thin] (b) -- (B11);
            \draw[very thin] (b) -- (B12);
            
            \coordinate (C) at ($(V6) + ({0, 0})$);
            \draw [black,fill=red] (C) ellipse (0.35 and 0.15);
            \coordinate (Ca1) at ($(C) + ({0.35*cos(170)}, {0.15*sin(170)})$);
            \coordinate (Ca2) at ($(C) + ({0.35*cos(30)}, {0.15*sin(30)})$);
            \coordinate (Cb1) at ($(C) + ({0.35*cos(230)}, {0.15*sin(230)})$);
            \coordinate (Cb2) at ($(C) + ({0.35*cos(40)}, {0.15*sin(40)})$);
            
            \draw[very thin] (a1) -- (Ca1);
            \draw[very thin] (a1) -- (Ca2);
            \draw[very thin] (b1) -- (Cb1);
            \draw[very thin] (b1) -- (Cb2);
          \end{tikzpicture}
          \hspace{0.4cm}
          \begin{tikzpicture}[scale = 0.9]
            \coordinate (V1) at (-0.85,0);
            \coordinate (V2) at (0.85,0);
            \coordinate (V7) at (-1.6,-1);
            \coordinate (V3) at (1.6,-1);
            \coordinate (V6) at (-1.45,-2);
            \coordinate (V4) at (1.45,-2);
            \coordinate (V5) at (0,-2.75);
            \draw (V1) ellipse (0.6 and 0.3);
            \draw (V2) ellipse (0.6 and 0.3);
            \draw (V3) ellipse (0.6 and 0.3);
            \draw (V4) ellipse (0.6 and 0.3);
            \draw (V5) ellipse (0.6 and 0.3);
            \draw (V6) ellipse (0.6 and 0.3);
            \draw (V7) ellipse (0.6 and 0.3);
            \coordinate (VV1) at ($(V1) + (-0.85, 0)$);
            \coordinate (VV2) at ($(V2) + (0.85, 0)$);
            \coordinate (VV3) at ($(V3) + (0.85, 0)$);
            \coordinate (VV4) at ($(V4) + (0.85, 0)$);
            \coordinate (VV5) at ($(V5) + (0.85, 0)$);
            \coordinate (VV6) at ($(V6) + (-0.85, 0)$);
            \coordinate (VV7) at ($(V7) + (-0.85, 0)$);
            \draw (VV1) node {\footnotesize $V_1''$};
            \draw (VV2) node {\footnotesize $V_2''$};
            \draw (VV3) node {\footnotesize $V_3''$};
            \draw (VV4) node {\footnotesize $V_4''$};
            \draw (VV5) node {\footnotesize $V_5''$};
            \draw (VV6) node {\footnotesize $V_6''$};
            \draw (VV7) node {\footnotesize $V_7''$};
            
            \coordinate (S) at (0, -1.3);
            \draw (S) ellipse (0.4 and 0.35);
            \coordinate (SS) at ($(S) + (-0.7, 0)$);
            \draw (SS) node {\footnotesize $S''$};
            
            \coordinate (a) at ($(S) + ({0},{0})$);
            \coordinate (b) at ($(V1) + ({0},{0})$);
            \draw[black, very thick] (a) -- (b);
            \draw (a) node[fill=black,circle,minimum size=2pt,inner sep=0pt] {};
            \draw (b) node[fill=black,circle,minimum size=2pt,inner sep=0pt] {};
            \draw (a) node[left] {\footnotesize $a$};
            \draw (b) node[left] {\footnotesize $b$};
            
            \coordinate (A1) at ($(V4) + ({0, 0})$);
            \draw [black,fill=gray!20] (A1) ellipse (0.4 and 0.23);
            \coordinate (A11) at ($(A1) + ({0.4*cos(220)}, {0.23*sin(220)})$);
            \coordinate (A12) at ($(A1) + ({0.4*cos(70)}, {0.23*sin(70)})$);
            \coordinate (a1) at ($(A1) + ({-0.1, 0})$);
            \coordinate (B1) at ($(V2) + ({0, 0})$);
            \draw [black,fill=gray!20] (B1) ellipse (0.4 and 0.23);
            \coordinate (B11) at ($(B1) + ({0.4*cos(105)}, {0.23*sin(105)})$);
            \coordinate (B12) at ($(B1) + ({0.4*cos(255)}, {0.23*sin(255)})$);
            \coordinate (b1) at ($(B1) + ({0.2, 0})$);
            
            \draw (a1) node[fill=black,circle,minimum size=2pt,inner sep=0pt] {};
            \draw (b1) node[fill=black,circle,minimum size=2pt,inner sep=0pt] {};
            \draw (a1) node[right] {\footnotesize $a'$};
            \draw (b1) node[left] {\footnotesize $b'$};
            \draw[very thin] (a) -- (A11);
            \draw[very thin] (a) -- (A12);
            \draw[very thin] (b) -- (B11);
            \draw[very thin] (b) -- (B12);

            \coordinate (C) at ($(V3) + ({0, 0})$);
            \draw [black,fill=red] (C) ellipse (0.35 and 0.15);
            \coordinate (Ca1) at ($(C) + ({0.35*cos(185)}, {0.15*sin(185)})$);
            \coordinate (Ca2) at ($(C) + ({0.35*cos(340)}, {0.15*sin(340)})$);
            \coordinate (Cb1) at ($(C) + ({0.35*cos(185)}, {0.15*sin(185)})$);
            \coordinate (Cb2) at ($(C) + ({0.35*cos(20)}, {0.15*sin(20)})$);
            
            \draw[very thin] (a1) -- (Ca1);
            \draw[very thin] (a1) -- (Ca2);
            \draw[very thin] (b1) -- (Cb1);
            \draw[very thin] (b1) -- (Cb2);
          \end{tikzpicture}
          \caption{Proof of \cref{prop:not bipartite}: Case 1 for $j=i+2$ (left), Case 1 for $j=i+3$ (middle) and Case 2 for $j=i+3$ (right). The red part is the common neighborhood of $a$ and $b$ (or $a'$ and $b'$).\label{fig:non-bipartite proof}}
        \end{figure}
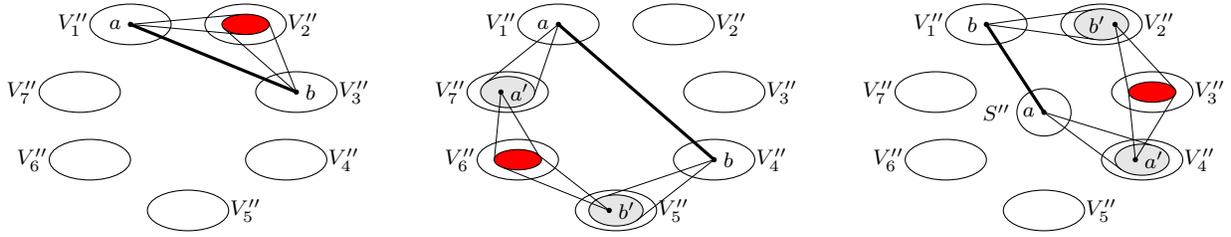
    \noindent   Propositions~\ref{prop:bipartite} and~\ref{prop:not bipartite} imply the theorem. 
    \end{proof}
    \section{Concluding remarks and open questions}
        It would be interesting to determine the possible values of $\polythres(H)$ for $3$-chromatic graphs $H$. So far we know that $\frac{1}{2k+1}$ is a value for each $k \geq 1$. Is there a graph $H$ with $\frac{1}{5} < \polythres(H) < \frac{1}{3}$? Also, is it true that $\polythres(H) > \frac{1}{5}$ if $H$ is not homomorphic to $C_5$?

        Another question is whether the inequality in \Cref{thm:poly-hom} is always tight, i.e. is it always true that $\polythres(H) = \delta_{\text{hom}}(\mathcal{I}_H)$?
        
        Finally, we wonder whether the parameters $\polythres(H)$ and $\linearthres(H)$ are monotone, in the sense that they do not increase when passing to a subgraph of $H$. We are not aware of a way of proving this without finding $\polythres(H),\linearthres(H)$.

    \bibliographystyle{abbrv} 
    \bibliography{library}

\end{document}